\documentclass[a4paper,10pt]{article}
\usepackage{LatexDefinitions}
\usepackage{color, booktabs}
\usepackage{enumitem}
\usepackage[T1]{fontenc}
\usepackage[nottoc]{tocbibind}

\usepackage[normalem]{ulem}

\floatstyle{ruled}
\newfloat{algorithmfloat}{t}{lop}
\floatname{algorithmfloat}{Algorithm}

\newcommand{\proj}{\tn{P}}
\newcommand{\cont}{\mc{C}}

\newcommand{\LL}{\mathrm{L}}
\newcommand{\CC}{\mathrm{C}}

\usepackage{esint}

\newcommand{\partA}{\mc{J}_A}
\newcommand{\partB}{\mc{J}_B}

\newcommand{\partk}{\mc{J}_k}
\newcommand{\partGeneric}{\mc{J}}

\newcommand{\Lebesgue}{\mathcal{L}}

\newcommand{\flows}{\mc{F}}

\newcommand{\pihat}{\hat{\pi}}

\newcommand{\nuhat}{\hat{\nu}}

\newcommand{\neigh}{\mathcal{N}}

\newcommand{\cbar}{\bar{c}}

\newcommand{\piInit}{\pi_{\rm init}}
\newcommand{\Jhat}{\hat{J}}
\newcommand{\bmpi}{\bm{\pi}}
\newcommand{\bmomega}{\bm{\omega}}

\newcommand{\DomDecMPI}{\textsc{DomDecMPI}}
\newcommand{\DomDecGPU}{\textsc{DomDecGPU}}
\newcommand{\SinkhornGPU}{\textsc{SinkhornGPU}}
\newcommand{\partBatch}{\tilde{\partGeneric}}
\newcommand{\indexBatch}{\tilde{I}}

\author{Ismael Medina, Bernhard Schmitzer}
\title{Flow updates for domain decomposition of entropic optimal transport}
\begin{document}
\maketitle
\begin{abstract}
Domain decomposition has been shown to be a computationally efficient distributed method for solving large scale entropic optimal transport problems. However, a naive implementation of the algorithm can freeze in the limit of very fine partition cells (i.e.~it asymptotically becomes stationary and does not find the global minimizer), since information can only travel slowly between cells. In practice this can be avoided by a coarse-to-fine multiscale scheme. In this article we introduce flow updates as an alternative approach. Flow updates can be interpreted as a variant of the celebrated algorithm by Angenent, Haker, and Tannenbaum, and can be combined canonically with domain decomposition. We prove convergence to the global minimizer and provide a formal discussion of its continuity limit. We give a numerical comparison with naive and multiscale domain decomposition, and show that the flow updates prevent freezing in the regime of very many cells.
While the multiscale scheme is observed to be faster than the hybrid approach in general, the latter could be a viable alternative in cases where a good initial coupling is available.
Our numerical experiments are based on a novel GPU implementation of domain decomposition that we describe in the appendix.
\end{abstract}
\section{Introduction}
\label{sec:introduction}
\subsection{Motivation}
\label{sec:motivation}

\paragraph{(Computational) optimal transport.} Optimal transport is a fundamental optimization problem with profound connections to various branches of mathematics, and applications in image analysis and machine learning. Given two probability measures $\mu$ and $\nu$ on spaces $X$ and $Y$ and a measurable cost function $c: X\times Y \rightarrow \R$, the optimal transport problem consists in finding the optimal probability measure $\pi$ in the set of transport plans $\Pi(\mu, \nu)$ --- probability measures on $X \times Y$ with marginals $\mu$ and $\nu$ --- that solves
\begin{equation}
	\label{eq:unregularized-ot}
	\inf_{\pi \in \Pi(\mu, \nu)}
	\int_{X\times Y} c(x,y) \diff \pi(x,y).
\end{equation}
Thorough expositions can be found in \cite{Villani-OptimalTransport-09} and \cite{SantambrogioOT}. The field has seen immense progress in the development of numerical algorithms, such as solvers for the Monge--Amp\`ere equation \cite{ObermanMongeAmpere2014}, semi-discrete methods \cite{LevySemiDiscrete2015,KiMeThi2019}, and multiscale methods \cite{MultiscaleTransport2011,SchmitzerSchnoerr-SSVM2013}.
An introduction to computational optimal transport, an overview on numerical methods, and applications can be found in \cite{PeyreCuturiCompOT}.

An important variant of the above problem is to add entropic regularization of the transport plan, i.e.~solving
\begin{equation}
\inf_{\pi\in\Pi(\mu, \nu)}
\int_{X\times Y} c(x,y) \diff \pi(x,y) 
+
\veps \KL(\pi | \mu \otimes \nu)
\end{equation}
where $\KL$ is the Kullback-Leibler divergence (see \eqref{eq:KL} for a definition) and $\veps>0$ is a positive regularization parameter.

This problem has very convenient analytical and computational advantages with respect to \eqref{eq:unregularized-ot}: it is strictly convex, differentiable with respect to the marginal distributions, and comes with an efficient, relatively simple, and GPU friendly numerical solver: the Sinkhorn algorithm, see for instance \cite{Cuturi2013,PeyreCuturiCompOT,FeydyDissertation} and references therein. 
A mature library for GPU implementation is presented in \cite{feydy_keops,feydy_geomloss}.
Exemplary references for the analysis of the limit $\veps \to 0$ are \cite{Cominetti-ExpBarrierConvergence-1992,LeonardSchroedingerMK2012,Carlier-EntropyJKO-2015}.
Entropic regularization is also useful for the analysis of stochastic processes and for stochastic optimal control, see for instance \cite{Chen2016,Lavenant2021,BarLav21}.

\paragraph{Domain decomposition for optimal transport.} Large optimal transport problems can be addressed with domain decomposition. Following \cite{BenamouPolarDomainDecomposition1994, BoSch2020}, the strategy of the algorithm is the following: let $\partA$ and $\partB$ be two partitions of $X$, which need to overlap in a suitable sense. Then, given an initial feasible coupling $\pi^0 \in \Pi(\mu, \nu)$, optimize its restriction to each set of the form $X_J\times Y$, where $X_J$ are the subdomains in the partition $\partA$, leaving the marginals on each of the restrictions fixed. 
That is, defining $\pi_J$ as the restriction of $\pi^0$ to $X_J \times Y$, and $\mu_J$ and $\nu_J$ respectively the $X$- and $Y$- marginals of $\pi_J$, the subplan $\pi_J$ is replaced by the optimizer of 
	\begin{equation}
	\nonumber
		\min_{\pi \in \Pi(\mu_J,\iter{\nu_J}{})}
		\int_{X_J \times Y} c\,\diff \pi + \veps\,\KL(\pi|\mu_J \otimes \nu).
	\end{equation}
This can be done independently for each subdomain in the partition, so the algorithm can be easily parallelized. Then repeat this procedure on partition $\partB$, and keep alternating between partitions until convergence.

Convergence of domain decomposition for unregularized optimal transport \eqref{eq:unregularized-ot} with the squared Eulidean distance cost is shown in \cite{BenamouPolarDomainDecomposition1994} and \cite{asymptotic_domdec_arxiv}, under different assumptions on the overlap between partitions. For entropy-regularized optimal transport, \cite{BoSch2020} shows linear convergence under minimal assumptions on the cost, marginal and partitions. Besides, \cite{BoSch2020} proposes an efficient, parallel, multiscale implementation for two-dimensional grids that outperforms a state-of-the-art sparse Sinkhorn solver even with a single worker.
Indeed, while the concept of domain decomposition (and the flow updates introduced in this article) is in principle applicable to unstructured point clouds in any dimension, the method will be most practical on low-dimensional grids, where an efficient partition structure is readily available.

In \cite{asymptotic_domdec_springer} we studied the asymptotic behavior of domain decomposition in the limit of small subdomains, with the motivation of providing an analogous analysis to what \cite{Berman-Sinkhorn-2017} obtained for the Sinkhorn algorithm.
Let $X=[0,1]^d$, let $(\pi^{n,k})_k$ be the domain decomposition iterates where the subdomains are hypercubes of edge length $1/n$, and define $\bmpi^n(t) \assign \pi^{n, \lfloor nt \rfloor}$, that is, each iterate is assigned a time step of $1/n$. Then under suitable assumptions, the trajectories $(\bmpi^n)_n$ converge to a limit trajectory, which solves the horizontal continuity equation
\begin{align*}
	\partial_t \bmpi_t + \ddiv_X \bmomega_t = 0 \qquad \tn{for $t>0$ and}
	\qquad \bmpi_{t=0} = \piInit
\end{align*}
where the horizontal momentum field $\bmomega_t$ describes how mass moves in $X$, while retaining the same position in $Y$ and $\ddiv_X$ is the corresponding divergence operator. $\bmomega_t$ can be obtained from an asymptotic version of the domain decomposition subdomain problem. As opposed to \cite{Berman-Sinkhorn-2017}, however, in the asymptotic limit not all initializations evolve to the optimal coupling but may become stationary in sub-optimal states. This is referred to as \emph{freezing}.

\paragraph{Freezing in domain decomposition.} Even though under the assumptions of \cite{BenamouPolarDomainDecomposition1994}, \cite{BoSch2020} or \cite{asymptotic_domdec_arxiv} every sequence of iterates $(\pi^{n,k})_k$ must converge to the optimal coupling, the rate may deteriorate with increasing number of subdomains even in smooth problem instances.
A striking example of this (see Figure \ref{fig:freezing} for an illustration) is when the initial coupling is given by $\pi^0=(\id,T)_{\#} \mu$ for some non-optimal transport map $T$, such that $T$ has non-zero curl, e.g.~when $T$ contains some \emph{non-local rotation}. Indeed, the number of required iterations for domain decomposition to undo such a non-local rotation increases with the number of partition cells, the intuitive reason being that information can only travel by one cell per iteration.
Based on the experiments in \cite{BoSch2020}, this freezing behaviour does not seem to be an issue when domain decomposition is used in combination with a multiscale scheme, as with the latter, macroscopic curl is removed from the initial coupling efficiently at a coarse resolution scale and thus at the finer levels essentially only local updates to the coupling are necessary.
However, as of now, there is no theoretical description of this behaviour.

\begin{figure}[hbt]
	\centering
	\includegraphics[width=\linewidth]{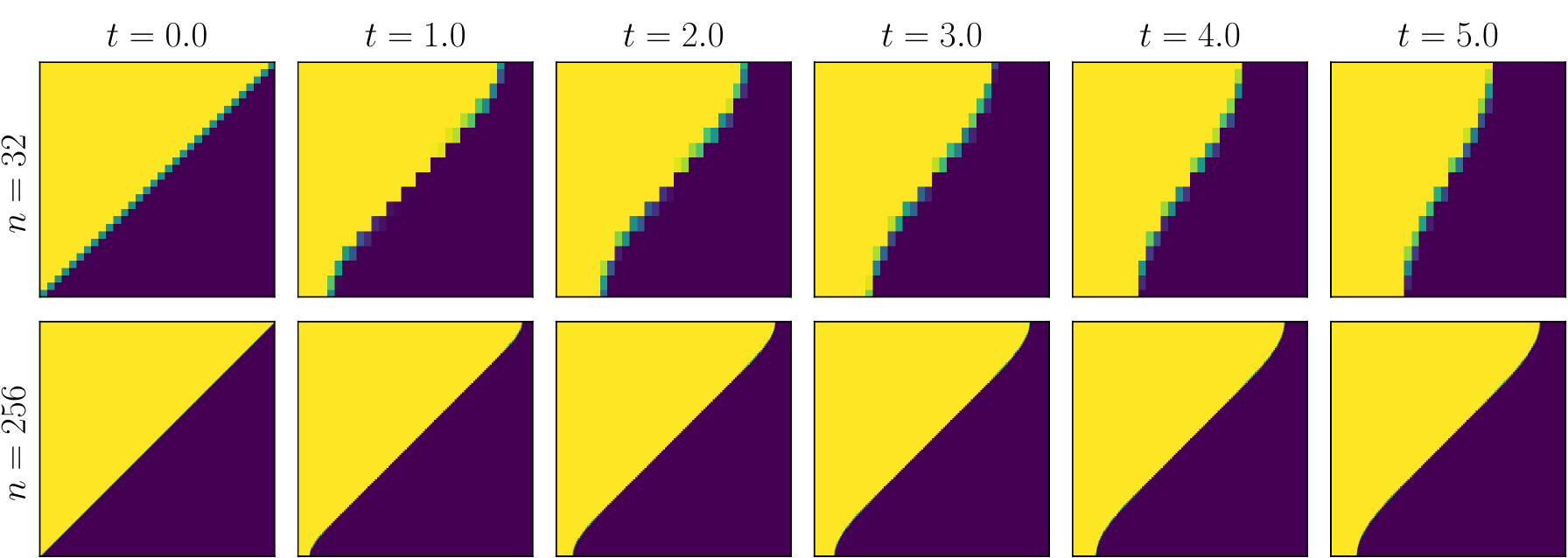}
	\caption{An example of the freezing behaviour in domain decomposition. For $\mu = \Lebesgue \restr [-1/2,+1/2]^2$, $\nu = \tfrac{1}{2} [\delta_{(-1/2, 0)} + \delta_{(+1/2, 0)}]$, we visualize the couplings $\pi^{n,k}$ by colouring the domain decomposition cells that are mapped to $(-1/2, 0)$ in yellow and those mapped to $(+1/2, 0)$ in dark blue; non-deterministic assignments feature an intermediate color.
	We show two domain decomposition trajectories at different resolutions, for an initialization that features a global rotation with respect to the optimal configuration, in this case corresponding to a vertical interface. 
	Both trajectories evolve towards the optimal coupling, but the convergence rate deteriorates drastically as the resolution increases.}
	\label{fig:freezing}
\end{figure}

\paragraph{Curl in optimal transport and the AHT scheme.} The relation between curl and optimal transport is well known since the seminal work of Brenier \cite{MonotoneRerrangement-91}, who showed that for  $\mu$ absolutely continuous with respect to the Lebesgue measure $\Lebesgue$ (in fact, a slightly weaker but more technical condition is shown to be sufficient in \cite{MonotoneRerrangement-91}), any mass rearrangement map $T$ can be uniquely decomposed as $T = \nabla \phi \circ s$, for some convex function $\phi$ and a $\mu$-measure preserving map $s$, i.e.~$s_\# \mu = \mu$, where $s_\#\mu$ denotes the pushforward of measure $\mu$ by the map $s$. $\nabla\phi$ is then the optimal transport map from $\mu$ to $T_\#\mu$ with respect to the squared Euclidean distance.

Based on this insight, Angenent, Haker, and Tannenbaum proposed an algorithm \cite{OptimalTransportWarpingTheory} (also known as AHT scheme) for computing the optimal map $\nabla\phi$.
Given $\mu$ and $\nu$, starting from some initial feasible transport map $T_0$ with $T_{0\#}\mu=\nu$, one (formally) constructs a flow of maps
\begin{equation}
\label{eq:AHTFlowGeneral}
T_t = T_0 \circ s_t^{-1} \qquad \tn{with} \qquad s_0=\id, \qquad \partial_t s_t = v_t \circ s_t,
\qquad \ddiv(v_t \cdot \mu)=0
\end{equation}
where $v_t$ is a Eulerian velocity field and the divergence constraint ensures that $s_{t\#}\mu=\mu$ and therefore $T_{t\#}\mu=\nu$ at all times $t$. Among admissible velocity fields, $v_t$ is chosen such that the transport cost
\begin{equation}
\label{eq:AHTObjectiveDerivative}
E(T) \assign \int_X c(x,T(x))\,\diff \mu(x)
\end{equation}
is decreasing in time. Formally one finds for $T_t$ as in \eqref{eq:AHTFlowGeneral} that
$$\partial_t E(T_t) = \int_X \langle \nabla_X c(x,T_t(x)), v_t(x) \rangle \diff \mu(x).$$
The notion of steepest descent then depends on the choice of a metric on the space of velocity fields.
For instance, the steepest descent with respect to the $\LL^2(\mu)$-metric would be given by setting $v_t$ to the projection of the vector field $x \mapsto -\nabla_X c(x,T_t(x))$ onto the subspace satisfying $\ddiv(v_t \cdot \mu)=0$ in $\LL^2(\mu)$.
In \cite{OptimalTransportWarpingTheory} a slightly different choice is proposed: One reparametrizes $v_t(x)=u_t(x)/\mu(x)$ (where $\mu(x)$ denotes the Lebesgue density of $\mu$ at $x$). $u_t$ is then set to be the $\LL^2(\Lebesgue)$-projection of $x \mapsto -\nabla_X c(x,T_t(x))$ onto the subspace satisfying $\ddiv(u_t \cdot \Lebesgue)=0$. This can be computed efficiently by solving a standard Poisson equation.

When $c$ is the squared Euclidean distance, at a critical point of $E$ the map $T_t$ must be a gradient of some potential $\phi$, and if the Jacobian of $T_t$ happens to be positive semi-definite, $\phi$ will be convex, i.e.~$T_t$ must be the sought-after optimal transport map.

Despite its elegance, the AHT scheme is in general not guaranteed to be well-posed (i.e.~the initial value problem  \eqref{eq:AHTFlowGeneral} where $v_t$ is some $\LL^2$-projection of $-\nabla_X c(\id,T_t)$ is not necessarily well-posed) or to converge to the optimal transport map, because for stationary $T_t=\nabla \phi$, $\phi$ will not necessarily be convex (except in particular cases and under strong assumptions, see \cite[Proposition 6.3.]{SantambrogioOT}, \cite[Section 1.3]{Brenier2009_AHT}).
Optimization based on re-arrangements $s_t$ will also not be applicable to the case of general transport plans $\pi \in \Pi(\mu,\nu)$ that are not concentrated on graphs of maps, as for instance in entropic transport.

However, the AHT scheme provides a powerful intuition as well as appealing properties, such as the score being non-increasing, with the rate of decrement of the score related to the curl component of the current transport map. 

\subsection{Contributions and outline}
\label{sec:contribution}
\paragraph{A hybrid scheme.} We observe that the strengths of domain decomposition and the AHT scheme are complementary. On the one hand, (entropy regularized) domain decomposition converges to the globally optimal plan but is slow at resolving curl. On the other hand, the AHT scheme removes curl quickly, but is not guaranteed to be well-posed or to achieve optimality. In this article we present a hybrid of the two methods that will always converge to the global minimizer and is efficient at removing curl.

Several issues must be resolved for this combination. In particular the AHT scheme needs to be discretized consistently at the level of the domain decomposition partitions $X_J$, and it must be adapted to work with entropic plans from domain decomposition, whereas the original AHT scheme was based on transport maps for unregularized transport.

\paragraph{Flow updates.} In Section \ref{sec:hybrid-scheme} we present a hybrid scheme for (entropic) optimal transport where domain decomposition updates are intertwined with \textit{flow updates}, which can be interpreted as a min-cost flow adaptation of the AHT scheme. Flow updates and their basic properties (preserving the marginals and decreasing the objective) are given in Sections \ref{sec:hybrid-scheme-simple} and \ref{sec:hybrid-scheme-general}.
Combination with the domain decomposition updates and proof of convergence to the optimal coupling are given in Section \ref{sec:hybrid-scheme-convergence}.

\begin{figure}[hbt]
	\centering
	\includegraphics[width =\textwidth]{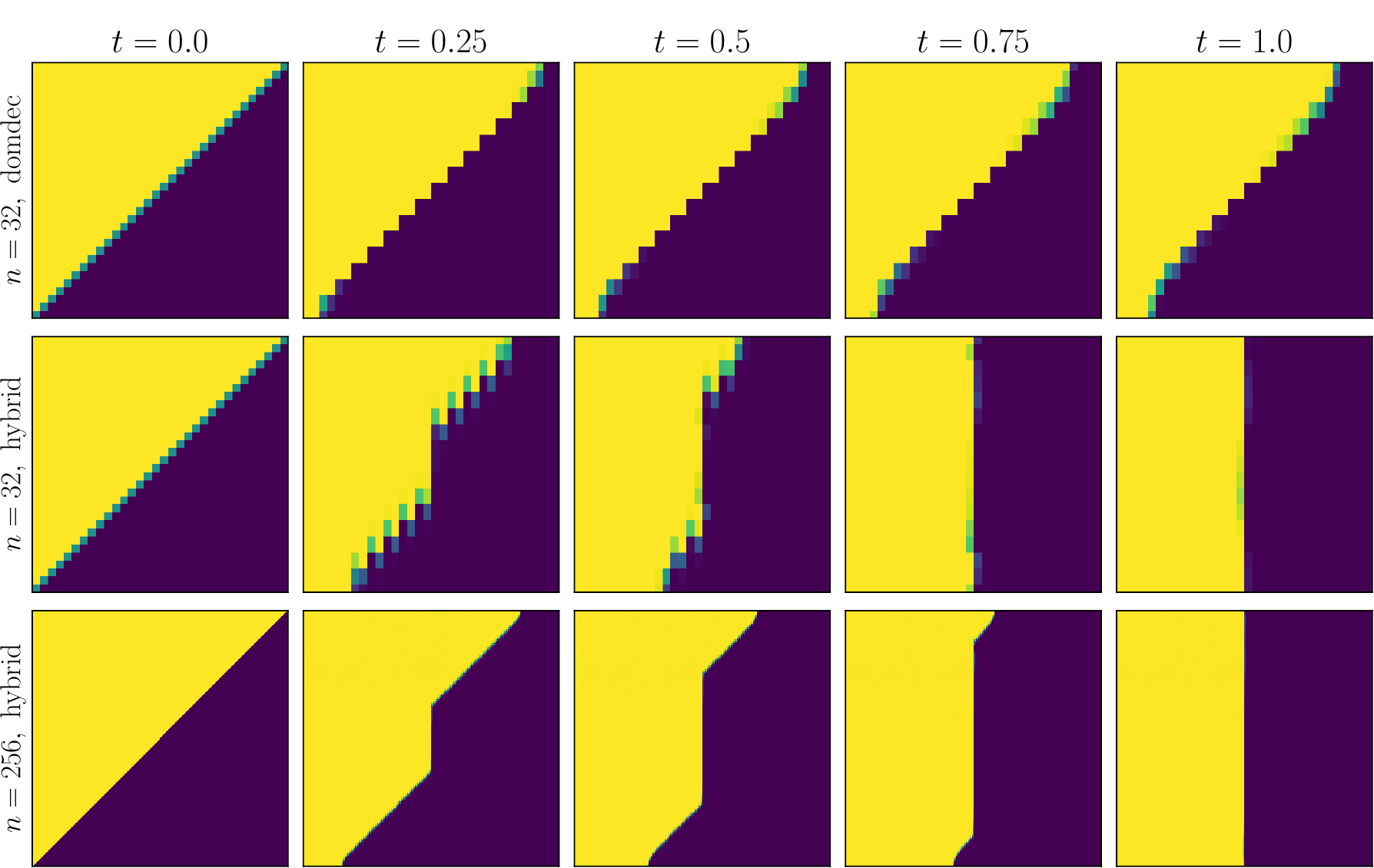}
	\caption{
		Hybrid scheme overcomes freezing:
		The top row shows the same sequence of iterates as Figure \ref{fig:freezing}, now from $t = 0$ until $t = 1$, showcasing the freezing. 
		The center row shows the corresponding trajectory for the hybrid scheme iterates: after the same number of iterations the iterates get much closer to the global optimizer, leaving only small local perturbations of the optimal assignment that can quickly be resolved by a few additional domain decomposition iterations.
		Finally, the bottom row shows a higher resolution example, converging in approximately the same time as the one with lower resolution, demonstrating the resilience to freezing.
	}
	\label{fig:domdechybridn64}
\end{figure}

\paragraph{Interpretation of flow updates as an $\LL^\infty$-version of the AHT scheme.}
In Section \ref{sec:continuity-limit} we show that flow updates can formally be interpreted as a discretized $\LL^\infty$-version of the AHT scheme, i.e.~in \eqref{eq:AHTObjectiveDerivative} one chooses the admissible $v_t$ that yields the largest decrement, subject to the constraint $|v_t(x)|_\infty \leq 1$ for all $x$ where $|\cdot|_\infty$ denotes the supremum norm on $\R^d$.
The domain decomposition iterates then deal with general non-deterministic transport plans $\pi$ and with discretization artefacts.
Our discussion remains purely formal to provide an intuitive interpretation. We anticipate that a rigorous derivation will be non-trivial and far beyond the scope of this article.

\paragraph{Numerical experiments.} The ability of the hybrid scheme to overcome the freezing behaviour is illustrated in Figure \ref{fig:domdechybridn64} and examined experimentally in more detail in Section \ref{sec:numerics}.
We demonstrate that the hybrid scheme clearly improves upon domain decomposition at a single scale.
On the other hand we observe that multiscale domain decomposition does not suffer from freezing and is in general faster than the hybrid scheme. However, the hybrid scheme could be a viable alternative in cases where a good initial coupling is known.

\paragraph{A GPU implementation.}
Our numerical experiments are based on a novel GPU implementation of domain decomposition that we describe in Appendix \ref{sec:gpu}.

Background on entropic optimal transport and domain decomposition is briefly recalled in Section \ref{sec:background}.

\section{Background}
\label{sec:background}
\subsection{Setting and notation}

\begin{itemize}
	\item Let $\R_+$ = $[0, \infty)$.
	\item Let $X$ and $Y$ be compact metric spaces. Typically, $X$ and $Y$ are convex, compact subsets of $\R^d$, but our scheme can be applied to more general spaces. We assume compactness to avoid overly technical arguments while covering the numerically relevant setting. 
	\item For a compact metric space $Z$ we denote by $\cont(Z)$ the continuous function on $Z$ and by $\meas(Z)$ the set of signed Radon measures over $Z$ (and finite by compactness of $Z$).
	We identify the latter with the topological dual of the former, and recall that integration against test functions in $\cont(Z)$ induces the weak* topology on $\meas(Z)$.
	The subset of non-negative and probability measures are $\measp(Z)$ and $\prob(Z)$. 
	The Radon norm of $\mu \in \meas(Z)$ is denoted by $||\mu||_{\meas(Z)}$, and it holds $||\mu||_{\meas(Z)} = \mu(Z)$ for $\mu \in \meas_+(Z)$. If there is no ambiguity we will simply write $||\mu||$.
	\item For $\mu, \nu \in \meas(Z)$ we write $\mu \ll \nu$ to indicate that $\mu$ is absolutely continuous with respect to $\nu$. 
	\item For $\mu$ in $\measp(Z)$ and $S \subset Z$ measurable the restriction of $\mu$ to $S$ is denoted by $\mu \restr S$.
	\item We denote by $\la \phi, \rho \ra$ the integration of a measurable function $\phi$ against a measure $\rho$. 
	\item The maps $\proj_X$ and $\proj_Y$ denote the projections of measures on $X\times Y$ to its marginals, i.e.
	\begin{align*}
		(\proj_X \pi)(S_X) \assign \pi(S_X \times Y) \qquad \tn{and} \qquad (\proj_Y \pi)(S_Y) \assign \pi(X \times S_Y)
	\end{align*}
	for $\pi \in \measp(X \times Y)$, $S_X \subset X$, $S_Y \subset Y$ measurable.
	\item For a compact metric space $Z$ and measures $\mu\in \meas(Z)$, $\nu \in \measp(Z)$, the \textit{Kullback-Leibler divergence} of $\mu$ with respect to $\nu$ is given by 
	\begin{align}
	\KL(\mu|\nu) \assign \begin{cases}
		\int_Z \varphi\left(\RadNik{\mu}{\nu}\right)\,\diff \nu & \tn{if } \mu \ll \nu,\,\mu \geq 0, \\
		+ \infty & \tn{else,}
	\end{cases}
	\quad
	\tn{with} \;\;
	\varphi(s) \assign \begin{cases}
		s\,\log(s)-s+1 & \tn{if } s>0, \\
		1 & \tn{if } s=0.
	\end{cases}
	\label{eq:KL}
	\end{align}
\item For $p \in [1,\infty)$ we denote by $|v|_p$ the $p$-norm on $\R^d$ and $|v|_\infty \assign \max_{i=1,\ldots,d} |v_i|$.
\end{itemize}

In Section \ref{sec:hybrid-scheme-convergence} we will need the following lemma. 
\begin{lemma}[Continuity of the restriction operator, Proposition 8.4.4 in \cite{bogachev-measure-theory}]
	\label{lemma:continuity-restriction}
	Let $(\rho_n)_n \subset \meas_+(Z)$ converging weak* to $\rho$. Let $A\subset Z$ closed,  
	and assume that $\rho(A) = \lim_{n\to\infty} \rho_n(A)$.
	Then $\rho\restr A = \lim_{n\rightarrow\infty}\rho_n \restr A$.
\end{lemma}

\subsection{Entropic optimal transport}
For $\mu \in \measp(X)$ and $\nu \in \measp(Y)$ we define the set of \textit{transport plans} or \textit{couplings} between $\mu$ and $\nu$ as 

\begin{equation}
\Pi(\mu, \nu)
\assign 
\{\pi \in \measp(X \times Y) 
\mid 
\proj_X\pi = \mu,
\quad 
\proj_Y\pi = \nu\}.
\end{equation}

Let $c$ be a lower-semicontinuous, positive, bounded function on $X\times Y$ and $\veps \in \R_+$. 
The entropic optimal transport problem between $\mu$ and $\nu$, with cost $c$ and regularization strength $\varepsilon$
is given by:

\begin{equation}
\label{eq:entropic-transport}
\min_{\pi\in\Pi(\mu, \nu)}
\la c, \pi \ra 
+
\veps \KL(\pi | \mu \otimes \nu)
\end{equation}

We call the first term in \eqref{eq:entropic-transport} the \textit{transport objective} and the second the \textit{entropic objective}. 
For $\veps = 0$ one recovers the Kantorovich optimal transport problem; existence of solutions is covered for example in \cite[Theorem 4.1]{Villani-OptimalTransport-09}. 
Some properties of the minimizers of the entropic problem are given below:

\begin{proposition}[Optimal entropic transport couplings]\leavevmode
	\label{prop:properties-entropic-transport}
	\begin{enumerate}[label=(\roman*)]
		\item \label{item:uniqueness-entropic-ot}
		\eqref{eq:entropic-transport} has a unique minimizer $\pi^* \in \Pi(\mu, \nu)$.
		\item \label{item:potentials-unique-up-to-constant}
		There exist measurable $\alpha^*: X\rightarrow \R$, $\beta^*:Y \rightarrow \R$ such that $\pi^*  = e^{(\alpha^* \oplus \beta^* - c)/\veps} \mu\otimes \nu$.
		The \textit{entropic potentials }
		$\alpha^*$ and $\beta^*$ are unique $\mu$-a.e.~and $\nu$-a.e.~up to a constant offsets, i.e. $(\alpha^*,\beta^*) \to (\alpha^* + C, \beta^* - C)$ for $C\in \R$.
		\item \label{item:characterization-entropic-ot}
		A transport plan $\pi^\ast \in \Pi(\mu, \nu)$ of the form $\pi^\ast = e^{(\alpha^* \oplus \beta^* - c)/\veps}\mu\otimes \nu$ with $\alpha^\ast \in \LL^\infty(X,\mu)$, $\beta^\ast \in \LL^\infty(Y,\nu)$ is optimal for \eqref{eq:entropic-transport}.
	\end{enumerate}
\end{proposition}
For proofs see \cite[Propositions 2.3 and 2.5]{BoSch2020} and references therein.

\subsection{Domain decomposition for optimal transport}
\label{sec:domdec}

Domain decomposition for optimal transport was originally proposed in \cite{BenamouPolarDomainDecomposition1994} and studied in \cite{BoSch2020} for entropic transport.
We recall here the main definitions. From now on, $\mu, \nu$ will be two probability measures, respectively on $X$ and $Y$.

\begin{definition}[Basic and composite partitions]\hfill
\label{def:partitions}
\begin{enumerate}
	\item For some finite index set $I$, let $\{X_i\}_{i\in I}$ be partition of $X$ into closed, disjoint subsets (up to a set of $\mu$-zero measure), were $m_i \assign \mu(X_i)$ is positive for all $i\in I$. We call $\{X_i\}_{i\in I}$ the \textit{basic partition}, and $\{m_i\}_{i\in I}$ the \textit{basic cell masses}. The restriction of $\mu$ to basic cells, i.e. $\mu_i \assign \mu\restr X_i$ for $i\in I$, are called the \textit{basic cell $X$-marginals}.
	\item A \textit{composite partition} is a partition of $I$. For each $J$ in a composite partition $\partGeneric$, we define the \textit{composite cells}, and the \textit{composite cell $X$-marginals} respectively by 
	\begin{align*}
		X_J & \assign \bigcup_{i \in J} X_i, 
		&
		\mu_J & \assign \sum_{i \in J} \mu_i=\mu \restr X_J.
	\end{align*}
	We will consider two composite partitions $\partA$ and $\partB$. 
\end{enumerate}
\end{definition}
A prototypical choice for basic and composite partitions for $X=[0,1]^2$ is given by an interleaving grid of cubes as sketched in Figure \ref{fig:basicandcompositecells}.
Convergence of domain decomposition to the optimal coupling hinges on some form of connectivity property of these partitions, where the precise notion of connectivity depends on the setting (e.g.~without or with entropic regularization, see \cite{BenamouPolarDomainDecomposition1994,asymptotic_domdec_arxiv,BoSch2020}). The key requirement is that optimality of the coupling on each of the composite cells implies global optimality of the coupling.
For entropic optimal transport, connectedness of the \emph{partition graph} (see Definition \ref{def:PartitionGraph}) is sufficient and necessary \cite{BoSch2020}.

\begin{figure}[hbt]
	\centering
	\includegraphics[width=0.8\linewidth]{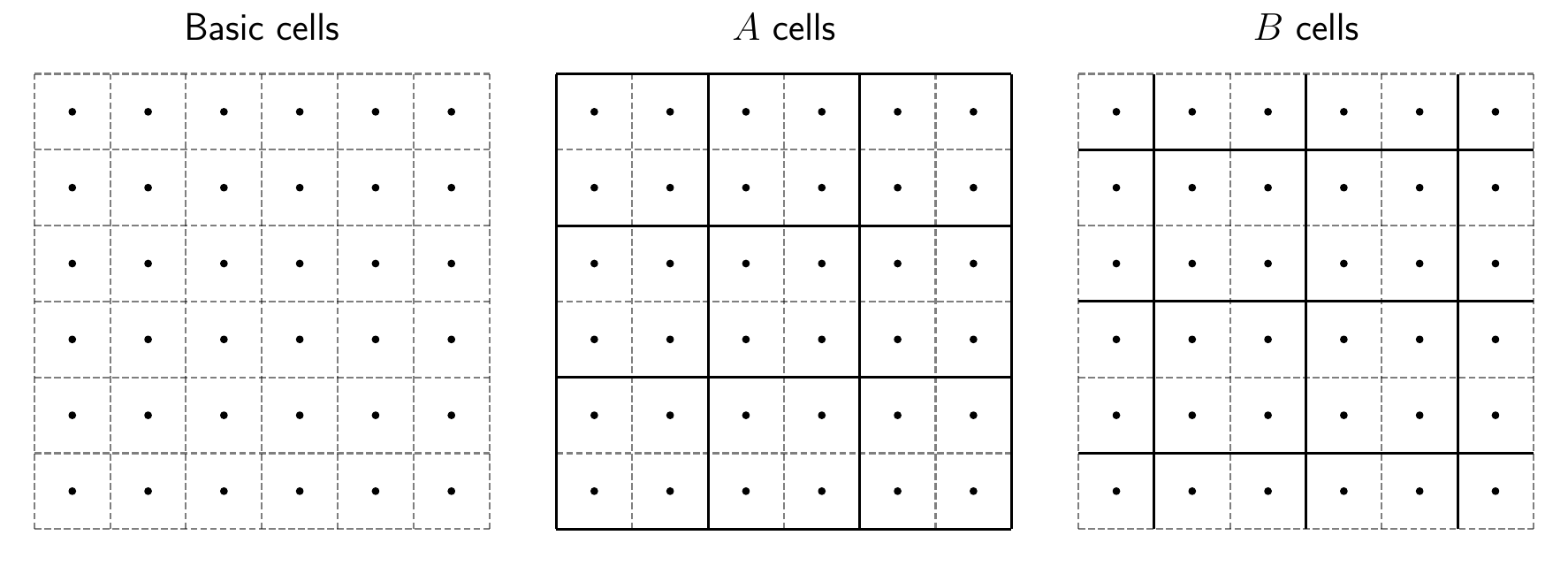}
	\caption{Prototypical choice of basic and composite partitions for $X = [0,1]^2$. For a resolution scale $n \in \N$, $X$ is first partitioned into \textit{basic cells} of size $1/n \times 1/n$. The restriction of $\mu$ to every basic cell is approximated by a Dirac delta in the center of the cell (indicated by the dots). The $A$ and $B$ partitions are constructed by $2\times 2$ groups of basic cells, with an offset between $A$ and $B$ groups.}
	\label{fig:basicandcompositecells}
\end{figure}

The domain decomposition algorithm is now stated in Algorithm \ref{alg:DomDecIter} and \ref{alg:DomDec}, where we will reuse the sub-routine of Algorithm \ref{alg:DomDecIter} in the hybrid scheme.
One may notice that the subproblem in line \ref{codeline:cell-ot-problem} of Algorithm \ref{alg:DomDecIter} is slightly different from \eqref{eq:entropic-transport}, since the reference measure in the $\KL$ term is not the product of the problem marginals.
This reference measure is inherited by restricting the global entropic optimal transport objective \eqref{eq:entropic-transport} to the cell $X_J \times Y$.
Nevertheless, noting that $0 \le \nu_J \le \nu$, a direct computation shows that the two problems are equal up to a constant contribution, and therefore Proposition \ref{prop:properties-entropic-transport} also applies to the domain decomposition cell subproblem. Keeping $\mu_J \otimes \nu$ as reference measure in line \ref{codeline:cell-ot-problem} emphasizes the interpretation of domain decomposition as a block coordinate descent method that decreases the global objective in every step.

\begin{algorithmfloat}[hbt]
	\noindent
	\textbf{Input}: current coupling $\pi \in \Pi(\mu,\nu)$ and partition $\partGeneric$.
	
	\noindent
	\textbf{Output}: new coupling $\pi'\in\Pi(\mu,\nu)$
	\smallskip
	
	\begin{algorithmic}[1]
		\ForAll{$J \in \partGeneric$}
		\Comment{iterate over each composite cell}
		\State $\mu_J \leftarrow \proj_X(\pi \restr(X_{J} \times Y))$
		\Comment{retrieve $X$-marginal on cell}
		\State $\iter{\nu_J}{}\, \leftarrow \proj_Y(\pi \restr(X_{J} \times Y))$
		\label{codeline:get-composite-marginal}
		\Comment{compute $Y$-marginal on cell}
		\State $\iter{\pi_{J}}{} \leftarrow \arg\min \left\{
		\int_{X_J \times Y} c\,\diff \pi + \veps\,\KL(\pi|\mu_J \otimes \nu) \,\middle|\, \pi \in \Pi\left(\mu_J,\iter{\nu_J}{}\right)\right\}$ 
		\label{codeline:cell-ot-problem}
		\EndFor
		\State $\pi' \leftarrow \sum_{J \in \partGeneric} \iter{\pi}{}_{J}$
	\end{algorithmic}
	\caption{\textsc{DomDecIter} \protect{\cite[Algorithm 1]{BoSch2020}}}
	\label{alg:DomDecIter}
\end{algorithmfloat}

\begin{algorithmfloat}[hbt]
	\noindent
	\textbf{Input}: initial coupling $\piInit \in \Pi(\mu,\nu)$
	
	\noindent
	\textbf{Output}: a sequence $(\iter{\pi}{k})_{k}$ of feasible couplings in $\Pi(\mu,\nu)$
	\smallskip
	
	\begin{algorithmic}[1]
		\State $\iter{\pi}{0} \leftarrow \piInit$
		\State $k \leftarrow 0$
		\Loop
		\State $k \leftarrow k+1$
		\State \algorithmicif\ ($k$ is odd)\ \algorithmicthen\ $\partk \leftarrow \partA$\ \algorithmicelse\ $\partk \leftarrow \partB$
		\Comment{select the partition}
		\State $\iter{\pi}{k} \leftarrow \text{\textsc{DomDecIter}}(\pi^{k-1}, \partk)$
		\Comment{solve all subproblems}
		\EndLoop
	\end{algorithmic}
	\caption{Domain decomposition for optimal transport \protect{\cite[Algorithm 1]{BoSch2020}}}
	\label{alg:DomDec}
\end{algorithmfloat}

Given the iterates $(\pi^k)_k$ computed in Algorithm \ref{alg:DomDec}, we call 
	\begin{equation}
		\label{eq:basic-cell-Y-marginal}
		\nu_i^k 
		\assign 
		\proj_Y( \pi^k \restr (X_i\times Y)),
		\quad
		i\in I.
	\end{equation}
	the \textit{basic cell $Y$-marginals}.
	In fact, for a performant implementation one should merely store $(\nu_i^k)_{i\in I}$ instead of the full $\pi^k$, since they require less memory and are sufficient to compute $\nu_J$ in line \ref{codeline:get-composite-marginal} of Algorithm \ref{alg:DomDecIter}, and (parts of) the full iterates can be reconstructed efficiently from the basic cell $Y$-marginals wherever required.
	For more details on the implementation of the domain decomposition algorithm we refer to \cite[Section 6]{BoSch2020}.

\section{Hybrid scheme}
\label{sec:hybrid-scheme}

As discussed in the introduction, domain decomposition performs poorly on initializations that contain a substantial nonlocal curl component. In this Section we show how to enhance the domain decomposition iteration with an additional global update step designed to reduce this curl, which we call the \textit{flow update}. 

We start in Section \ref{sec:hybrid-scheme-simple} with a simplified setting to build intuition. Section \ref{sec:hybrid-scheme-general} defines the flow update in full generality and establishes that it decreases the entropic optimal transport objective \eqref{eq:entropic-transport} while preserving the marginals. Finally, Section \ref{sec:hybrid-scheme-convergence} introduces a scheme that interleaves entropic domain decomposition with flow updates, and shows convergence of the iterates to the optimal coupling.
Numerical experiments demonstrating the efficiency of the new scheme are presented in Section \ref{sec:numerics}.

\subsection{Flow update for singleton basic cells}
\label{sec:hybrid-scheme-simple}
For now, assume that $X$ is finite and the basic partition $I$ divides $X$ into singletons, i.e.~each $X_i$ consists of a single $x_i$ for all $i \in I$.
We assume that a directed graph with vertex set $I$ and edge (or neighbourhood) set $\neigh \subset I \times I$ is given, satisfying the following definition.
\begin{definition}
	\label{def:neigh-function}
	A \emph{neighborhood set} is a set $\neigh \subset I \times I$ with the properties:
	\begin{enumerate}
		\item $(i,i) \in \neigh$ for all $i \in I$.
		\item For all $i,j \in I$, $(i,j) \in \neigh$ if and only if $(j,i) \in \neigh$.
	\end{enumerate}
	We write $\neigh(i) \assign \{ j \in I | (i,j) \in \neigh\}$.
\end{definition}
That is, the directed graph $(I,\neigh)$ is symmetric and each vertex has an edge to itself. On this graph we can introduce the set of mass preserving flows.
\begin{definition}
For a neighbourhood set $\neigh$, a tuple $w \assign (w_{ij})_{(i,j) \in \neigh}$, $w_{ij} \in \R_+$ for $(i,j) \in \neigh$, is a \emph{(mass preserving) flow} if it satisfies the divergence constraint
\begin{equation}
\label{eq:div-constraint-single-cell}
	\sum_{j \in \neigh(i)} w_{ij}
	=
	\sum_{j \in \neigh(i)} w_{ji}
	= 
	\mu(X_i) \assignRe m_i\quad \tn{for all $i\in I$.} 
\end{equation}
For $(i,j) \in \neigh$ we interpret $w_{ij}$ as the amount of mass flowing from cell $i$ to $j$, in particular $w_{ii}$ denotes the mass that remains in cell $i$.
Denote by $\flows$ the set of mass preserving flows (omitting the dependency on $I$, $\neigh$, and $\mu$ in the notation for simplicity).
\end{definition}

Consider now some coupling $\pi \in \Pi(\mu,\nu)$ and recall the basic cell $Y$-marginals $\nu_i \assign \proj_Y (\pi\restr\{x_i\}\times Y)$ for each $i\in I$, \eqref{eq:basic-cell-Y-marginal}.
Note that $\pi = \sum_{i \in I} \delta_{x_i} \otimes \nu_i$.
A mass preserving flow $w \in \flows$ induces a new coupling via
\begin{align}
	\label{eq:PiHatSimple}
	\hat{\nu}_i & \assign \sum_{j \in \neigh(i)} w_{ji} \frac{\nu_j}{m_j}
	\qquad \tn{for } i \in I, &
	\hat{\pi} & \assign \sum_{i \in I} \delta_{x_i} \otimes \hat{\nu}_i.
\end{align}
Using \eqref{eq:div-constraint-single-cell} it is easy to verify that $\hat{\pi} \in \Pi(\mu,\nu)$ (see Section \ref{sec:hybrid-scheme-general} for a proof in a more general setting).
For the change in the transport objective we find
\begin{align}
\la c, \hat{\pi}-\pi \ra & = \sum_{i \in I} \la c(x_i,\cdot), \hat{\nu}_i - \nu_i \ra
= \sum_{i \in I} \la c(x_i,\cdot), \sum_{j \in \neigh(i)} w_{ji} \frac{\nu_j}{m_j}
- \sum_{j \in \neigh(i)} w_{ij} \frac{\nu_i}{m_i} \ra \nonumber \\
& = \sum_{(i,j) \in \neigh} w_{ij} \la c(x_j,\cdot)-c(x_i,\cdot), \frac{\nu_i}{m_i} \ra.
\end{align}
Minimizing this over all flows $w \in \flows$ (i.e.~finding the best new coupling $\hat{\pi}$ that can be reached by a flow from $\pi$) corresponds to solving the min-cost flow problem
\begin{equation}
\label{eq:MinCostFlow}
\min_{w \in \flows} \sum_{(i,j) \in \neigh} w_{ij} \cdot \cbar_{ij}
\end{equation}
for cost coefficients
\begin{equation}
\label{eq:CostCoefficientSimple}
\cbar_{ij}	\assign \la c(x_j,\cdot)-c(x_i,\cdot), \frac{\nu_i}{m_i} \ra
\qquad \tn{for } (i,j) \in \neigh.
\end{equation}
Note that \eqref{eq:MinCostFlow} can be interpreted as an optimal transport problem on the graph $(I, \neigh)$.
We call the procedure of replacing $\pi$ by $\hat{\pi}$ as given by \eqref{eq:PiHatSimple} with a minimizing flow $w$ of \eqref{eq:MinCostFlow}, a \emph{flow update}.
Key properties of the flow update are:
\begin{itemize}
\item \textbf{The problem \eqref{eq:MinCostFlow} is global}, i.e.~it will be able to detect and remove curl that may be barely detectable at the level of individual composite cells.
	\item \textbf{The optimal flow $w$ can be computed efficiently} for moderate sizes of $I$ and sparse $\neigh$, since there are highly efficient solvers for the min-cost flow problem.
	\item \textbf{The flow update does not modify the marginals.} This is a consequence of the divergence constraints \eqref{eq:div-constraint-single-cell}.
	\item \textbf{The flow update does not increase the transport objective nor the entropy objective.}
	For the former this is rather easy to see, since $\cbar_{ii}=0$ and an admissible flow is given by keeping all mass in its original location, i.e.~$w_{ij}=0$ for $i \neq j$.
	The latter is not as obvious, but it hinges on convexity of the entropy and the fact that the $\hat{\nu}_i$ are (up to normalization) convex combinations of the $\nu_i$. Proofs for both are given in Section \ref{sec:hybrid-scheme-general}.
	\item \textbf{The flow update can formally be interpreted as a discretization of an $\LL^\infty$-version of the AHT scheme.} To anticipate this, note that in the limit of very fine partitions, $c(x_j,\cdot)-c(x_i,\cdot)$ formally becomes the partial derivative of $c$ in $X$ at $x_i$ in direction $x_j-x_i$ (here $X \subset \R^d$), $\nu_i/m_i$ becomes the disintegration of $\pi$ with respect to its $X$-marginal at $x_i$, so that $\cbar_{ij}$ becomes the partial derivative at $x_i$, averaged over $Y$ with respect to the disintegration. Finally, $w_{ij}$ formally becomes a momentum measure, or $w_{ij}/m_i$ becomes a velocity field, that describes the horizontal movement of mass and the divergence constraints \eqref{eq:div-constraint-single-cell} enforce that the velocity field has its components bounded by 1 almost everywhere. We revisit this intuition in Section \ref{sec:continuity-limit}.
\end{itemize}
On the other hand, \textbf{for large $X$ the above strategy will not be computationally practical,} since solving problem \eqref{eq:MinCostFlow} will become inefficient.
As a compromise, in Section \ref{sec:hybrid-scheme-general}, we will adapt the flow update to the setting where the basic cells $X_i$ are not singletons. This will require a more complex interpretation of flows $w$ as induced plans $\hat{\pi}$, \eqref{eq:PiHatSimple}, and a corresponding adjusted definition of the cost coefficients $\cbar_{ij}$, \eqref{eq:CostCoefficientSimple}.
Intuitively, the flow update will then implement a global optimization of $\pi$ at the resolution level of the basic cells $(X_i)_{i \in I}$ (which can be kept at a moderate size), whereas domain decomposition will locally optimize on the scale of individual cells.

\subsection{Flow update for general basic cells}
\label{sec:hybrid-scheme-general}
If basic cells $X_i$ are no longer singletons, construction \eqref{eq:PiHatSimple} is no longer applicable to translate a flow into an update of $\pi$.
Instead, we will need a more general instruction on how to arrange mass from cell $i$ in cell $j$, while preserving the marginals of the coupling. This role is played by \emph{edge candidates}.

\begin{definition}[Edge candidates]
	\label{def:edge-candidates}
	A candidate for edge $(i,j) \in \neigh$ is a coupling $\pihat_{ij} \in \Pi(\mu_j/m_j,\nu_i/m_i)$.
\end{definition}

We will specify in Definition \ref{def:glued-plan} and below how edge candidates can be constructed in a practical way. Once a set of such candidates is fixed, we can define the general flow update.

\begin{definition}[Flow update]
	\label{def:flow-update}	
	Let $\pi \in \measp(X \times Y)$, $(\pi_i)_{i\in I}$ its decomposition into basic cells, with cell masses $(m_i)_{i\in I}$, cell $Y$-marginals $(\nu_i)_{i\in I}$ and cell $X$-marginals $(\mu_i)_{i\in I}$. Let $\neigh$ be a neighborhood set and $(\pihat_{ij})_{(i,j)\in\neigh}$ be edge candidates, with $\pihat_{ii} = \pi_i/m_i$ for each $i\in I$. Define the edge cost coefficients
	\begin{equation}
	\label{eq:edge-cost}
	\cbar_{ij} \assign \la c, \pihat_{ij}-\pihat_{ii} \ra.
	\end{equation}
	A \textit{flow update} is a coupling of the form 
	\begin{equation}
		\label{eq:flow-update-new-plan}
		\pihat \assign \sum_{j\in I} \pihat_j
		,
		\quad
		\tn{with }
		\pihat_j 
		\assign 
		\sum_{i\in \neigh(j)}
		w_{ij} \cdot \pihat_{ij}
	\end{equation}
	where $(w_{ij})_{ij}$ is a solution to the min-cost flow problem \eqref{eq:MinCostFlow} with cost coefficients	 $(\cbar_{ij})_{(i,j) \in \neigh}$.
\end{definition}

\begin{remark}
If the basic cells are singletons, i.e.~$X_i = \{x_i\}$ for $i\in I$, as above, then the only feasible candidate for each edge is $\pihat_{ij} \assign \delta_{x_j} \otimes \nu_i/m_i$.
In this case \eqref{eq:edge-cost} becomes \eqref{eq:CostCoefficientSimple} and \eqref{eq:flow-update-new-plan} becomes \eqref{eq:PiHatSimple}.
\end{remark}
	
\begin{remark}
For a given $\pi$, the flow update might not be unique, since solutions to the min-cost flow problem \eqref{eq:MinCostFlow} might not be unique. This is however no issue for the algorithm.
\end{remark}

The following two lemmas show that a flow update does not change the marginals nor does it increase the transport score. 

\begin{lemma}
	\label{lemma:flow-preserves-marginals}
	A coupling $\pi\in \measp(X \times Y)$ preserves its global marginals under a flow update. The new cell $Y$-marginals are given by
	\begin{equation}
	\label{eq:flow-update-cell-marginals}
	\nuhat_j
	\assign 
	\sum_{i\in \neigh(j)} w_{ij} \cdot \frac{\nu_i}{m_i}.
	\end{equation}
\end{lemma}

\begin{proof}
	Let $\pi$ be the initial plan, and let $\pihat$ be a flow update as given in \eqref{eq:flow-update-new-plan}. Then, on each basic cell $j \in I$,
	\begin{equation*}
		\proj_X \pihat_j
		=
		\sum_{i\in \neigh(j)}
		w_{ij} \cdot \proj_X \pihat_{ij}
		=
		\sum_{i\in \neigh(j)}
		\frac{w_{ij}}{m_j} \cdot \mu_j
		=
		\proj_X \pi_j.
	\end{equation*}
	The new cell $Y$-marginals is given by
	\begin{equation*}
		\nuhat_j 
		=
		\proj_Y \pihat_j
		=
		\sum_{i\in \neigh(j)}
		w_{ij} \cdot \proj_Y \pihat_{ij}
		=
		\sum_{i\in \neigh(j)}
		w_{ij} \cdot \frac{\nu_i}{m_i},
	\end{equation*}
	and so the global $Y$-marginal is
	\begin{equation*}
		\proj_Y \pihat
		=
		\sum_{j\in I}
		\nuhat_j
		=
		\sum_{j\in I}
		\sum_{i\in \neigh(j)}
		\frac{w_{ij}}{m_i}\nu_i
		=
		\sum_{i\in I}
		\nu_i
		\sum_{j\in \neigh(i)}
		\frac{w_{ij}}{m_i}
		=
		\sum_{i\in I}
		\nu_i
		=
		\proj_Y \pi. \qedhere
	\end{equation*}
\end{proof}

\begin{lemma}
	\label{lemma:transport-non-increasing}
	The transport objective is non-increasing under a flow update. 
\end{lemma}

\begin{proof}
	Let $\pi$ be an initial plan, and $\pihat$ a flow update of $\pi$ where $w$ is the corresponding optimal flow of \eqref{eq:MinCostFlow}.
	By Definition \ref{def:flow-update} one has that
	$$\pihat=\sum_{(i,j) \in \neigh} w_{ij} \cdot \pihat_{ij} \qquad \tn{and} \qquad \pi=\sum_{(i,j) \in \neigh} w_{ij} \cdot \pihat_{ii}$$
	and therefore one obtains for the change in the transport objective
	\begin{align}
		\la c, \pihat-\pi \ra
		&=
		\sum_{(i,j) \in \neigh}
		w_{ij} \cdot \la c, \pihat_{ij}-\pihat_{ii} \ra =
		\sum_{(i,j) \in \neigh}
		w_{ij} \cdot \cbar_{ij}
		\label{eq:total-cost-update}
	\end{align}
	which is precisely the objective of \eqref{eq:MinCostFlow}.
	A feasible candidate for $w$ is the `stationary flow' that leaves all mass in its original cells, i.e.~$w_{ij}=m_i$ if $i=j$ and $w_{ij}=0$ otherwise. For this admissible choice one has that \eqref{eq:total-cost-update} equals zero, and thus the minimal value of \eqref{eq:MinCostFlow} must be non-positive, and therefore, so must be the change in the transport objective.
\end{proof}

The entropic score is not necessarily decreasing under a flow update for general edge candidates $\pihat_{ij}$. The following definition gives a practical way to craft edge candidates for which the entropic score is also non-increasing.

\begin{definition}[Gluing edge candidate]
	\label{def:glued-plan}
	Let $i\in I$, $j\in \neigh(i)\setminus \{i\}$ and $\gamma_{ij} \in \Pi(\mu_i/m_i, \mu_j/m_j)$. We refer to the plan $\pihat_{ij} \in \Pi(\mu_j/m_j,\nu_i/m_i)$ that is obtained by combining $\gamma_{ij}$ with $\pi_i/m_i \in \Pi(\mu_i/m_i,\nu_i/m_i)$ with the gluing lemma \cite[Lemma 7.6]{Villani-TOT2003} along their common $\mu_i/m_i$-marginal as \textit{gluing edge candidate} for edge $(i,j)$. $\pihat_{ij}$ acts on measurable sets $A \times B \subset X_j\times Y$ as follows:
	\begin{equation}
		\label{eq:glued-plan}
		\pihat_{ij}(A \times B)
		\assign 
		\frac{1}{m_i} \int_{X_i} 
		(\gamma_{ij})_{x}(A) \cdot
		(\pi_i)_{x}(B)\,
		\diff \mu_i(x)
	\end{equation}
	where $(\gamma_{ij})_{x}$ and $(\pi_i)_{x}$ denote the disintegrations of both measures against their $\mu_i/m_i$ or $\mu_i$ marginal at $x$ respectively. (Disintegrations are by convention probability measures, so the normalization of $\pi_i/m_i$ can be neglected at this point.)
\end{definition}
	
\begin{remark}
	\label{remark:eps-gamma}
	Two prototypical choices for $\gamma_{ij}$ are the product coupling $\mu_i/m_i \otimes \mu_j/m_j$, and an optimal transport plan (e.g.~for the squared distance on $X$).
	
	For the former one has $(\gamma_{ij})_{x}=\mu_j/m_j$ $\mu_i$-almost everywhere and thus \eqref{eq:glued-plan} simplifies to $\pihat_{ij} = \mu_j/m_j \otimes \nu_i / m_i$. This is straightforward to implement.
	However, the induced coupling candidates $\hat{\pi}$ in \eqref{eq:flow-update-new-plan} may be somewhat `bulky' and thus the associated cost coefficients $\cbar_{ij}$ might not be favourable for moving mass, slowing down the optimization scheme.
	
	The latter choice requires to pre-compute and store the plans $\gamma_{ij}$. This is feasible since basic cells are usually small and the set $\neigh$ is sparse. \eqref{eq:glued-plan} can then be implemented by a simple matrix multiplication.
	
	In numerical experiments (Section \ref{sec:numerics}) we find that both options perform well as long as $\pi$ is far from the optimum. Close to the optimum, the latter choice performs slightly better, at the cost of some additional computational effort. See Section \ref{sec:single-scale-experiments} for more details.

	One may also consider the optimal entropic transport plan between $\mu_i/m_i$ and $\mu_j/m_j$. The above choices then correspond to the limits of $\veps=\infty$ and $\veps=0$, respectively.
\end{remark}

\begin{lemma}[Properties of the gluing edge candidates] 
	\label{lemma:properties-gluing-candidates}
	For every $i\in I$, $j\in \neigh(i)\setminus \{i\}$ let $\gamma_{ij}$ be a transport plan between $\mu_i/m_i$ and $\mu_j/m_j$, and let $\pihat_{ij}$ be a gluing edge candidate following \eqref{eq:glued-plan}. Then:
	\begin{enumerate}
		\item $\pihat_{ij}$ is a candidate for edge $(i,j)$, i.e.~$\pihat_{ij} \in \Pi(\mu_j/m_j,\nu_i/m_i)$.
		\item If $\pi_i \ll \mu_i \otimes \nu$, then $\pihat_{ij} \ll \mu_j\otimes\nu$, and its density is given by
		\begin{equation}
			\label{eq:density-glued-candidate}
			\RadNikD{\pihat_{ij}}{(\mu_j\otimes \nu)}(x,y)
			=
			\frac{1}{m_j}
			\int_{X_i}
			\RadNikD{\pi_i}{(\mu_i\otimes \nu)}(x',y)
			\diff (\gamma_{ij})_x(x')
		\end{equation}
		$\mu_j\otimes\nu$-almost everywhere, where now $(\gamma_{ij})_x$ denotes the disintegration of $\gamma_{ij}$ at $x$ against the $\mu_j/m_j$-marginal.
		\item A flow update with edge candidates $(\pihat_{ij})_{ij}$ does not increase the entropic objective.
	\end{enumerate}
\end{lemma}

\begin{proof}
	Throughout this proof, variables $x$, $x'$ and $y$ run respectively over $X_j$ $X_i$, and $Y$.
	
	\textbf{Item 1:} This is a direct consequence of the gluing lemma and can also quickly be verified explicitly from \eqref{eq:glued-plan}, e.g.~for the $X$-marginal for measurable $A \subset X_j$ via
	\begin{align*}
		(\proj_X \pihat_{ij})(A)
		&=
		\pihat_{ij}(A \times Y)
		=	
		\frac{1}{m_i} \int_{X_i} 
		(\gamma_{ij})_{x}(A) \cdot
		(\pi_i)_{x}(Y)\,
		\diff \mu_i(x)
		=	
		\frac{1}{m_i} \int_{X_i} 
		(\gamma_{ij})_{x}(A)
		\diff \mu_i(x) \\
		& =\gamma_{ij}(A \times X_i) = \mu_j(A)/m_j.
	\end{align*}
	
	\textbf{Item 2:} Since $\RadNik{\pi_i}{\mu_i \otimes \nu}$ is non-negative, the quantity in \eqref{eq:density-glued-candidate} is unambiguously defined and we just need to check that it yields the same measure as \eqref{eq:glued-plan}. Let us integrate the right-hand side of \eqref{eq:density-glued-candidate} with respect to $\mu_j \otimes \nu$ on a set of the form $A \times B \subset X_j\times Y$:
	(Note that in the following computation we switch from the disintegration of $\gamma_{ij}$ against its $\mu_j/m_j$ marginal to the one against its $\mu_i/m_i$ marginal, denoted by $(\gamma_{ij})_x$ and $(\gamma_{ij})_{x'}$ respectively.)
	\begin{align*}
		\int_{A\times B} 
		&\left[
		\frac{1}{m_j}
		\int_{X_i}
		\RadNikD{\pi_i}{(\mu_i\otimes \nu)}(x',y)
		\diff (\gamma_{ij})_x(x')
		\right]
		\diff \mu_j(x) \diff \nu(y)
		=	
		\int_{X_i \times A \times B}
		\RadNikD{\pi_i}{(\mu_i\otimes \nu)}(x',y)
		\diff \gamma_{ij}(x',x) \diff \nu(y)
		\\
		&=
		\frac{1}{m_i}
		\int_{A \times X_i \times B}
		\RadNikD{\pi_i}{(\mu_i\otimes \nu)}(x',y)
		\diff (\gamma_{ij})_{x'}(x)
		\diff \mu_i(x') \diff \nu(y)
		=
		\frac{1}{m_i}
		\int_{A \times X_i \times B}
		\diff (\gamma_{ij})_{x'}(x)
		\diff\pi_i(x',y)
		\\
		&
		=
		\frac{1}{m_i}
		\int_{A \times X_i \times B}
		\diff(\gamma_{ij})_{x'}(x) 
		\diff(\pi_i)_{x'}(y) 
		\diff \mu_i(x')
		=
		\pihat_{ij}(A \times B).
	\end{align*}

	\textbf{Item 3:}
	If $\pi_i$ is singular with respect to $\mu_i\otimes \nu$ for some $i\in I$, the entropic score cannot increase.
	Otherwise, by the previous point $\pihat_{ij}$ has a density with respect to $\mu_i\otimes \nu$ given by \eqref{eq:density-glued-candidate}. Let us first compute how the relative entropy $\KL(m_j\,\pihat_{ij}|\mu_j\otimes \nu)$ relates to $\KL(\pi_i|\mu_i\otimes \nu)$:
	\begin{align}
		\KL(m_j\,\pihat_{ij}|\mu_j\otimes \nu)
		&=
		\int_{X_j \times Y}
		\varphi \left(m_j  
		\RadNikD{\pihat_{ij}}{(\mu_j\otimes \nu)}(x,y)
		\right)
		\diff \mu_j(x) \diff \nu(y)
		\nonumber
		\\
		&=
		\int_{X_j \times Y}
		\varphi \left(  
		\int_{X_i}
		\RadNikD{\pi_i}{(\mu_i\otimes \nu)}(x',y)
		\diff (\gamma_{ij})_x(x')
		\right)
		\diff \mu_j(x) \diff \nu(y).
		\nonumber
		\\
		\intertext{Since $(\gamma_{ij})_x$ is a probability measure and $\varphi$ is a convex function, we can use Jensen's inequality to bound this quantity:}
		&\le 
		\int_{X_j \times X_i \times Y}
		\varphi \left(  
		\RadNikD{\pi_i}{(\mu_i\otimes \nu)}(x',y)
		\right)
		\diff (\gamma_{ij})_x(x')
		\diff \mu_j(x) \diff \nu(y)
		\nonumber
		\\
		&=
		m_j\int_{X_j \times X_i \times Y}
		\varphi \left(  
		\RadNikD{\pi_i}{(\mu_i\otimes \nu)}(x',y)
		\right)
		\diff \gamma_{ij}(x',x)
		\diff \nu(y)
		\nonumber
		\\
		&=
		\frac{m_j}{m_i}
		\int_{X_i \times Y}
		\varphi \left(  
		\RadNikD{\pi_i}{(\mu_i\otimes \nu)}(x',y)
		\right)
		\diff \mu_i(x')
		\diff \nu(y)
		=
		\frac{m_j}{m_i}
		\KL(\pi_i|\mu_i\otimes \nu).
		\label{eq:estimate-KL-pihat}
	\end{align}
	 With this estimate, let us check that the global entropy does not increase after a flow update:
	\begin{align*}
	\label{eq:first-line-KL}
	\KL(\pihat | \mu \otimes \nu)
	&=
	\sum_{j\in I}
	\KL(\pihat_j | \mu_j \otimes \nu)
	=
	\sum_{j\in I}
	\KL
	( 
	\sum_{i\in \neigh(j)}
	w_{ij}\,\pihat_{ij}
	| 
	\mu_j \otimes \nu
	)
	\le 
	\sum_{(i,j) \in \neigh}
	\frac{w_{ij}}{m_j}
	\KL(m_j\,\pihat_{ij}| \mu_j \otimes \nu),
	\\
	\intertext{where in the last inequality we used Jensen's inequality in the first argument of the $\KL$. Now we use the estimate from \eqref{eq:estimate-KL-pihat}:}
	&\le 
	\sum_{(i,j)\in \neigh}
	\frac{w_{ij}}{m_j}
	\frac{m_j}{m_i}
	\KL(\pi_i|\mu_i\otimes \nu)
	=
	\sum_{i\in I}
	\left(
	\sum_{j\in \neigh(i)}
	\frac{w_{ij}}{m_i}\right)
	\KL(\pi_i|\mu_i\otimes \nu)
	=
	\KL(\pi | \mu \otimes \nu). \qedhere
	\end{align*}
\end{proof}

Finally, we summarize the previous results to collect the desired properties of the flow update.
\begin{proposition}
	\label{prop:flow-update-decreases-score}
	A flow update with gluing edge candidates preserves the marginals and has a non-increasing entropic optimal transport objective \eqref{eq:entropic-transport}.
\end{proposition}

\begin{proof} Let $\pi$ be the initial plan, $\pihat$ a flow update obtained with gluing edge candidates. Then Lemma \ref{lemma:flow-preserves-marginals} guarantees the marginals of $\pihat$ are those of $\pi$, Lemma \ref{lemma:transport-non-increasing} proves that $\la c, \pihat \ra \le \la c, \pi \ra$ and	Lemma \ref{lemma:properties-gluing-candidates} establishes that $\KL(\pihat | \mu\otimes\nu) \le \KL(\pi | \mu\otimes \nu)$. 
\end{proof}

\subsection{Hybrid scheme}
\label{sec:hybrid-scheme-convergence}

The flow updates described in the previous section are in general unable to generate a minimizing sequence on their own. For example, if $\pi = \mu \otimes \nu$, any flow update would yield again $\pihat = \mu \otimes \nu$, since $\pi$ already minimizes the relative entropy and by Lemma \ref{lemma:properties-gluing-candidates} a flow update cannot increase it (at least when using gluing edge candidates). Conversely, if $\pi$ is the optimal unregularized plan between $\mu$ and $\nu$, a flow update cannot bring it to the optimal entropic plan, since doing so would typically involve introducing a bit of blur, which will inevitably increase the transport score.

However, flow updates can be intertwined with domain decomposition steps to provide fast removal of global curl, on which domain decomposition performs poorly. The domain decomposition steps will in turn make progress where the flow updates would be stuck and guarantee convergence to a global minimizer.

We call the combination of domain decomposition and flow updates the \textit{hybrid scheme}. Algorithm \ref{alg:Hybrid} outlines a possible implementation, where two domain decomposition iterations are followed by a flow update. Of course different schedules can be considered (e.g.~running the flow update after only every $n$-th domain decomposition iteration), and they do not change significantly the properties of the algorithm. 

\begin{algorithmfloat}[h]
	\noindent
	\textbf{Input}: initial coupling $\piInit \in \Pi(\mu,\nu)$
	
	\noindent
	\textbf{Output}: a sequence $(\iter{\pi}{k})_{k}$ of feasible couplings in $\Pi(\mu,\nu)$
	\smallskip
	
	\begin{algorithmic}[1]
		\State $\iter{\pi}{0} \leftarrow \piInit$
		\State $k \leftarrow 0$
		\Loop
		\State $k \leftarrow k+1$
		\State \algorithmicif\ ($k\%3 == 1$)\ \algorithmicthen\
		$\iter{\pi}{k} \leftarrow \text{\textsc{DomDecIter}}(\pi^{k-1}, \partA)$
		\State \algorithmicelse  \algorithmicif\ ($k\%3 == 2$)\ \algorithmicthen\
		$\iter{\pi}{k} \leftarrow \text{\textsc{DomDecIter}}(\pi^{k-1}, \partB)$
		\State \algorithmicelse  \algorithmicif\ ($k\%3 == 0$)\ \algorithmicthen\
		$\pi^k \leftarrow \text{\textsc{FlowUpdate}}(\pi^{k-1})$
		\EndLoop
	\end{algorithmic} 
	\caption{Hybrid scheme for optimal transport}
	\label{alg:Hybrid}
\end{algorithmfloat}

Due to the monotonicity properties of the flow update, it is natural to expect the hybrid scheme to converge under the same assumptions as regular domain decomposition. In the case of entropic domain decomposition, apart from boundedness of the cost, the only other crucial assumption was connectedness of the \textit{partition graph} (see \cite[Section 4.3]{BoSch2020}). We give below a simplified version of the concept that is sufficient for our purposes: 

\begin{definition}[Partition graph, adapted from \protect{\cite[Definition 7]{BoSch2020}}]
\label{def:PartitionGraph}
		The partition graph is given by the vertex set $V \assign \partA \cup \partB$ and the edge set
		\begin{align*}
			E \assign \left\{ (J, \Jhat) \subset \partA \cup \partB \,
			\middle|
			\, J \cap\Jhat \neq \varnothing\right\},
 		\end{align*}
		i.e.~there is an edge between two composite cells if their intersection is non-empty. (Recall that the composite partitions $\partA$ and $\partB$ are partitions of the discrete index set $I$ for the basic partition cells of $X$.) In this way, connectedness of the partition graph intuitively translates to the ability of mass or information to travel between arbitrary subdomains.
		\label{item:PartitionGraph}
\end{definition}

\begin{proposition}[Convergence of the hybrid scheme]
	\label{prop:convergence-hybrid}
	Let $c$ be continuous, $\veps > 0$, consider partitions in the sense of Definition \ref{def:partitions} with strictly positive basic cell masses, and assume that the partition graph is connected. Let $(\pi^k)_k$ be a sequence obtained with Algorithm \ref{alg:Hybrid}, where the flow update uses gluing edge candidates (cf. Definitions \ref{def:flow-update}, \ref{def:glued-plan}). Then $(\pi^k)_k$ converges weak* to the unique minimizer of \eqref{eq:entropic-transport}.
\end{proposition}
For convenience we assume $c$ to be continuous here. See Remark \ref{rem:more-general-costs} for a discussion of more general costs.

\begin{proof}[Proof]
	Define by $S_A$ the map that applies an $A$ iteration to $\pi$, i.e. $\pi \mapsto$ \textsc{DomDecIter}$(\pi, \partA)$, analogously for $S_B$. These maps are continuous with respect to the weak* topology, since they are a composition of restriction of measures, projection and solution of entropic problems, which in our setting are all continuous (restriction by Lemma \ref{lemma:continuity-restriction}, solution by stability of entropic OT when the cost is continuous \cite{stability-entropic-ot}, projection is obvious). 
	Besides, these maps do not increase the score of a given coupling ---since domain decomposition is a coordinate descent algorithm---, and keep it feasible. 
	
	The iterates of Algorithm \ref{alg:Hybrid} thus verify: 	
	\begin{equation}
	\begin{aligned}
	\pi^{3\ell+1} = S_A(\pi^{3\ell}),
	\qquad&
	\pi^{3\ell+2} = S_B(\pi^{3\ell+1})
	\qquad &\tn{ for all } \ell\in \N, 
	\\
	\la c, \pi^{k+1} \ra 
	+
	\veps \KL(\pi^{k+1} | \mu \otimes \nu)
	\le\ 
	&
	\la c, \pi^{k} \ra 
	+
	\veps \KL(\pi^k | \mu \otimes \nu)
	\qquad &\tn{ for all } k \in \N,
	\end{aligned}
	\end{equation}
	the decrement for the flow updates being granted by \eqref{prop:flow-update-decreases-score}, because we are using gluing candidates.
	
	Since the couplings are supported on the compact space $X\times Y$, the sequence $(\pi^{3\ell})_\ell$ is tight, so it has weak* cluster points. Besides, by monotonicity of the score, all such cluster points share the same score, and they all lie in $\Pi(\mu,\nu)$ by continuity of the projection operator.
	
	Let $(\pi^{3\ell})_\ell$ be a subsequence (not relabeled) converging to one such cluster point $\pi^*$. Then, by continuity of the solve maps, $(\pi^{3\ell+2})_\ell = (S_B(S_A(\pi^{3\ell})))_\ell$ must converge to $S_B(S_A(\pi^*))$. This means that $\pi^*$ and $S_B(S_A(\pi^*))$ share the same score, since they are both cluster points of the original sequence. In turn, this implies that $\pi^*$ is a fixed point for both $S_A$ and $S_B$; otherwise applying $S_A$ or $S_B$ would have decreased its score.
	This follows from uniqueness of minimizers in Algorithm \ref{alg:DomDecIter}, line \ref{codeline:cell-ot-problem}, see (Prop.~\ref{prop:properties-entropic-transport}, item \ref{item:uniqueness-entropic-ot}).
	Therefore we conclude that $\pi^*$ is locally optimal on each subdomain of $\partA$ and $\partB$, which by Prop.~\ref{prop:properties-entropic-transport}, item \ref{item:characterization-entropic-ot} implies that there exist functions $\alpha_J$ on $X_J$ and $\beta_J$ on $Y$ such that
	\begin{equation}
		\label{eq:composite-cell-optimality}
		\pi^*_J 
		=
		\exp \left(\frac{\alpha_J(x) + \beta_J(y) - c(x,y)}{\veps}\right) \mu_J \otimes \nu, 
		\qquad
		\tn{for all } J \in \partA \cup \partB.
	\end{equation}
	
	Now we will build global entropic potentials from the local potentials in \eqref{eq:composite-cell-optimality}.
	Take an edge $(J, \Jhat)$ of the partition graph. Since the intersection of $J$ and $\Jhat$ is non-empty, there is some basic cell $i$ they share, which has positive mass by assumption. For $\mu\otimes \nu$ a.e. $(x,y)$ in $X_i \otimes Y$, the corresponding expressions \eqref{eq:composite-cell-optimality} for $J$ and $\Jhat$ must coincide, which means that the difference $\beta_J - \beta_{\Jhat}$  is a constant $\nu$-a.e. By connectedness of the partition graph, this property extends to every pair of composite cells, so we conclude that for every cell $J\in \partA \cup \partB$, $\beta_J = \beta^* - s_J$ $\nu$-a.e., where $\beta^* \assign \beta_{J_0}$ for some arbitrary cell $J_0$, and $(s_J)_{J\in\partA \cup \partB}$ are suitable real offsets.
	
	This means by Prop.~\ref{prop:properties-entropic-transport} \ref{item:potentials-unique-up-to-constant} that the potentials $(\alpha_J - s_J, \beta^*)$ also yield the same $\pi^*_J$ in \eqref{eq:composite-cell-optimality}. We can then obtain a global potential $\alpha^*$ in the following way: for each $x\in X$ find the cell $J\in \partA$ such that $x\in X_J$ (which is $\mu$-a.e. unique) and define 
	\begin{equation}
		\alpha^*(x) = \alpha_J(x) - s_J.
	\end{equation}
	
	Then $\pi^* = e^{(\alpha^* \oplus \beta^* - c)/\veps} \mu \otimes \nu$, since this new expression coincides with \eqref{eq:composite-cell-optimality} for every composite cell in $\partA$, and these constitute a partition of $X$. Finally, again by Prop.~\ref{prop:properties-entropic-transport} \ref{item:characterization-entropic-ot} this is the optimal coupling, this time for the global problem.
\end{proof}

\begin{remark}
	Note that in the proof of Proposition \ref{prop:convergence-hybrid} we only leverage the fact that the flow updates do not increase the score.
	The task of leveraging the structure of the flow updates to improve the convergence rate of domain decomposition seems daunting,
	since they can be interpreted as discretization of a gradient flow in a space where the objective is not geodesically convex (see end of Section \ref{sec:continuity-limit} for a detailed discussion).
	Of course, intuitively, we expect that flow updates and domain decomposition updates have complementary strengths and weaknesses (cf.~Section \ref{sec:contribution}). This intuition is supported by the numerical experiments in Section \ref{sec:numerics}.
\end{remark}

\begin{remark}[Convergence of the hybrid scheme for more general costs]
\label{rem:more-general-costs}
	As shown in \cite[Corollary 4.12]{BoSch2020} entropic domain decomposition converges to the optimal coupling for any bounded measurable cost function $c$. 
	We believe this to be the case also for the hybrid scheme, but to show it we should use a different stability result for the solve operator, in particular \cite[Theorem 4.3]{carlier-stability}, which assumes that the density of the marginals with respect to some fixed reference measure is bounded from above and away from zero. This is the case for domain decomposition after sufficiently many iterations, since the coupling will eventually have fully support due to the entropic regularization, and a lower bound on the density can be given (see \cite[Lemma 4.2, (ii)]{BoSch2020}). 
	However, a proof that this `diffusion' cannot be reverted by the flow updates would be overly technical and not adequate for the simple presentation we intend to give. At any rate, if the initialization $\pi^0$ has a density $ \RadNik{\pi^0}{\mu \otimes \nu}$ bounded away from zero, or if enough iterations are performed before the first flow update, all the subsequent iterations will satisfy the necessary density bounds. 
	The reason is that flow updates mix the marginals, and as such cannot worsen a uniform lower bound on the density.
\end{remark}

\section{Formal continuity limit and comparison with AHT scheme}
\label{sec:continuity-limit}
We now return to the setting of Section \ref{sec:hybrid-scheme-simple}. Let $n \in \N$ and let $X$ be a regular Cartesian grid in $[0,1]^d$ with $n$ grid points along each axis, let $\neigh$ be given by the standard $2d$-neighbourhood (e.g.~the 4-neighbourhood in $d=2$). We now formally discuss the limit $n \to \infty$.
We can identify a flow $w \in \flows$ with a vector measure $W \assign \sum_{(i,j) \in \neigh} w_{ij} \cdot n \cdot (x_j-x_i) \cdot \delta_{x_i} \in \meas(X)^d$. Note that $|x_j-x_i|_\infty =1/n$ for all $(i,j) \in \neigh$ with $i \neq j$.
Therefore, one has that $W \ll \mu = \sum_{i \in I} m_i\,\delta_{x_i}$ and $|\RadNik{W}{\mu}(x)|_\infty \leq 1$ for $\mu$-almost all $x$.
For $\phi \in \CC^1(X;\R^d)$ one then obtains
\begin{align*}
	\int_X \nabla \phi\,\diff W
	= \sum_{(i,j) \in \neigh} \langle \nabla \phi(x_i), x_j-x_i \rangle \cdot n \cdot w_{ij}
	= \sum_{(i,j) \in \neigh} \left[ n \cdot \phi(x_j)-n \cdot \phi(x_i)+o(1)\right] w_{ij}=o(1).
\end{align*}
Here we used the divergence constraints \eqref{eq:div-constraint-single-cell}, which also imply $\sum_{(i,j) \in \neigh} w_{ij}=1$, and $o(1)$ is to be understood with respect to the limit $n \to \infty$ and is obtained from the uniform continuity of $\nabla \phi$ on $X$. So as $n \to \infty$, any weak* cluster point of a sequence of measures $W$ will be divergence free in a distributional sense.
We interpret $\flows$ therefore as a discrete approximation of the weak* compact set
\begin{equation}
\label{eq:FormalVLimit}
\left\{v \cdot \mu \middle| v : [0,1]^d \to [-1,1]^d \tn{ measurable},
\ddiv(v \cdot \mu)=0
\right\}
\end{equation}
where $\ddiv$ denotes distributional divergence.
We do not expect that this discrete approximation will be dense in \eqref{eq:FormalVLimit} as $n \to \infty$, due to anisotropy artefacts of the grid graph. Fortunately, the combination with domain decomposition ensures that this is not a problem.
By similar arguments, for $c \in C^1(X \times Y)$, the min-cost flow objective \eqref{eq:MinCostFlow} can be written as
\begin{align*}
\sum_{(i,j) \in \neigh} w_{ij} \cdot \la c(x_j,\cdot)-c(x_i,\cdot), \frac{\nu_i}{m_i} \ra
= \int_X \la v(x), \int_Y \nabla_X c(x,y)\,\diff \pi_x(y) \ra \diff \mu(x) + o(1)
\end{align*}
where $v$ is, as above, the density of $W$ with respect to $\mu$, and we used that $\nu_i/m_i \in \prob(Y)$ is the disintegration of $\pi=\sum_{i \in I} \delta_{x_i} \otimes \nu_i$ with respect to its $X$-marginal at $x_i$.
Constructing a new coupling $\pihat$ via \eqref{eq:PiHatSimple} then corresponds to a time-step of size $1/n$.

Consider now the following minimization problem over measurable velocity fields:
\begin{equation}
\inf_{v : \ddiv(v \cdot \mu)=0} \int_X \la v(x), \int_Y \nabla_X c(x,y)\,\diff \pi_x(y) \ra \diff \mu(x) + \frac{1}{p} \int_X |v(x)|^p_p\,\diff \mu(x)
\label{eq:Linfty-AHT}
\end{equation}
For $p=2$ and $\pi=(\id,T_t)_\# \mu$ (i.e.~$\pi_x=\delta_{T_t(x)}$) the minimizing velocity field will be the $\LL^2(\mu)$-projection of $-\nabla_X c(\id,T_t)$ onto the constraint $\ddiv(v \cdot \mu)=0$, as discussed in \eqref{eq:AHTObjectiveDerivative} and below.
Conversely, the formal limit $p \to \infty$ will correspond to the minimization of the first term over the set \eqref{eq:FormalVLimit}, which is the formal continuity limit of our flow update.

Note that one also has $\Phi_{t\#} (\mu\otimes \nu)=\mu \otimes \nu$ and that formally $\Phi_t$ is invertible (if the flow of $v_t$ were well-posed) therefore $\KL(\pi_t|\mu\otimes \nu)=\KL(\Phi_{t\#}\pi_0|\Phi_{t\#}(\mu\otimes \nu))=\KL(\pi_0|\mu\otimes \nu)$. So formally the continuum limit of the flow updates keeps the entropy unchanged and the possible decrease that we observe at finite resolution is a discretization artefact.

In this sense, the flow update can formally be interpreted as a discretized $\LL^\infty$-version of the AHT scheme; i.e.,~as a gradient flow for the transport cost on the space of differentiable domain rearrangements. As we have just seen, the $\LL^\infty$-version allows for an exact discrete version of the divergence constraint as mass preserving flows on a discrete graph. The flow updates alone will in general not reach a global minimizer, due to anisotropy of the graph discretization and since it can only move fibers $\nu_i$ as a whole (recall the discussion at the beginning of Section \ref{sec:hybrid-scheme-convergence}).
The combination with domain decomposition iterations resolves both of these issues.

Note that the suboptimal stationary points of the AHT scheme remain stationary for \eqref{eq:Linfty-AHT}, including the $p \to \infty$ limit. Indeed, the constraint $\ddiv(v \cdot \mu)=0$ for $v$ implies: 
\begin{equation}
\int_X  \la v(x), \nabla \varphi(x) \ra \diff \mu(x) = 0
\quad 
\tn{for all $\varphi \in C^1(X)$.}
\end{equation}
This means that whenever the vector field $x\mapsto \int_Y \nabla_X c(x,y)\,\diff \pi_x(y)$ can be written as the gradient of a function $\varphi$, the optimal velocity field in \eqref{eq:Linfty-AHT} is identically zero. This has several implications. First, it shows that the $L^\infty$-variant would not improve the convergence behavior of the AHT scheme; the suboptimal stationary points remain the same (our motivation to consider it is the compatibility with the min-cost flow discretization). Second, we expect that the discrete version will suffer from related non-optimal points where the updates will be zero or at least very small. Consequently, it might be impossible to obtain a convergence rate for the hybrid scheme that is faster than pure domain decomposition.
This difficulty is caused by the fact that the transport cost is not geodesically convex within the space of admissible deformations (a challenge shared with the AHT scheme). A different relaxed version of the AHT scheme with entropic regularization is studied in \cite{multiflow}. It is shown that this relaxation has no sub-optimal stationary points and that it converges to the global minimizer. However, due to the aforementioned non-convexity, still no rate of convergence can be established.

\section{Numerical experiments}
\label{sec:numerics}
In this section we provide numerical experiments that show how flow updates can resolve the freezing behaviour of the domain decomposition algorithm, we study how the choice of the edge candidates $\hat{\pi}_{ij}$ influences the behaviour of the hybrid algorithm, and how flow updates relate to the multiscale scheme.

\subsection{Implementation notes}
\label{sec:implementation-notes}
A CPU multiscale implementation of domain decomposition, parallelized via MPI was presented in \cite{BoSch2020}.
In this article we apply a new GPU version, described in more detail in Appendix \ref{sec:gpu}.\footnote{Code for both CPU and GPU implementation available at \url{https://github.com/OTGroupGoe/DomainDecomposition}.}
The second component of our hybrid algorithm are the flow updates introduced in Sections \ref{sec:hybrid-scheme-simple} and \ref{sec:hybrid-scheme-general}. Their implementation can be divided into the following major steps:

\textbf{1. Compute the edge costs ${(\cbar_{ij})_{ij}}$, \eqref{eq:edge-cost}.} This requires edge candidates $(\pihat_{ij})_{ij}$ as constructed in \eqref{eq:glued-plan} which are constructed from transport plans $(\gamma_{ij})_{ij}$ from $\mu_i/m_i$ to $\mu_j/m_j$.
For $(\gamma_{ij})_{ij}$ we choose either an optimal entropic transport plan for some regularization strength $\veps_\gamma$ or the independent product plan $\mu_i/m_i \otimes \mu_j/m_j$ (which would correspond to the choice $\veps_\gamma=\infty$), see Remark \ref{remark:eps-gamma}.
Constructing $\pihat_{ij}$ from $\gamma_{ij}$ also requires to temporarily instantiate the measure $\pi_i$, which can lead to memory bottlenecks if this is done fully in parallel.

\textbf{2. Solve the min-cost flow problem \eqref{eq:MinCostFlow}.}
Many efficient solvers are available for this problem. In our experiments we use \texttt{OR-tools} \cite{ortools}. Popular alternatives are \texttt{LEMON} \cite{lemon} or \texttt{CPLEX} \cite{cplex}.

\textbf{3. Assemble the new basic cell $Y$-marginals via \eqref{eq:flow-update-cell-marginals}.} This is enough to perform a subsequent domain decomposition iteration. The new full coupling $\pi$ is not necessary.

\subsection{Single scale experiments}
\label{sec:single-scale-experiments}
As illustrated in Figure \ref{fig:freezing}, domain decomposition can asymptotically exhibit \emph{freezing} in non-optimal states in the limit of very fine cells and we propose flow updates as a remedy, as illustrated in Figure \ref{fig:domdechybridn64}. In this section we discuss a more complex numerical experiment to examine freezing, its resolution via flow updates, and the influence of the choice of edge candidates $(\pihat_{ij})_{ij}$, see Section \ref{sec:implementation-notes} and Remark \ref{remark:eps-gamma}.

\paragraph{Test problem and discretization.}
For our test problem we choose $\mu = \nu$ where $\mu$ is a measure on the square $[-1,1]^2$ with four symmetrical bumps, as pictured in Figure \ref{fig:four-bumps}, left.
In particular, $\mu$ has a 4-fold rotational symmetry, i.e.~for the counterclockwise rotation by $\pi/2$, given by
\begin{equation}
\label{eq:T0}
T_0: X\rightarrow Y,
\qquad 
[-1,1] \ni (x_1, x_2) \mapsto (-x_2, x_1)
\end{equation}
one has $T_{0\#}\mu=\mu$ and so $\pi_0 \assign (\id, T_0)_\# \mu$ is a non-optimal transport plan between $\mu=\nu$ and itself and we expect domain decomposition to freeze asymptotically when initialized with $\pi_0$, since the deviation from the optimal coupling $(\id,\id)_\# \mu$ is essentially `pure curl', making it an interesting test case for the hybrid scheme.

For numerical experiments the domain $[-1,1]^2$ and the measure $\mu$ are then discretized and represented by a uniform Cartesian grid with $N$ points along each axis.
Basic cells are then formed by squares of $s \times s$ pixels, corresponding to $n=N/s$ basic cells along each axis, each basic cell being a square with edge length $1/n=s/N$.
We choose $N \in \{32,128\}$ and $s \in \{2,8\}$ in this section and the concrete values will always be stated in the respective figures.
Finally, for comparing iterations across different resolutions we follow \cite{asymptotic_domdec_springer}, where we assign to each domain decomposition iterate a timestep $\Delta t = 1/n=s/N$ and define the measure trajectory $\R_+ \ni t \mapsto \pi_t \assign \pi^{\lfloor t/\Delta t\rfloor}$ where $\pi^{k}$ refers to the iterates of Algorithm \ref{alg:DomDec}, initialized with $\iterz{\pi} \assign \pi_0$. 
To keep the comparison consistent, for the hybrid scheme we will only track the number of domain decomposition iterations, or alternatively, we consider the flow update to be integrated at the end of the $B$-partition iteration. Hence we construct trajectories from the iterates of Algorithm \ref{alg:Hybrid} as follows: 
\begin{equation}
	\pi_t \assign \tilde{\pi}^{\lfloor t/\Delta t \rfloor},
	\quad\tn{with }
	\tilde{\pi}^{2\ell + m} = \pi^{3\ell + m},
	\quad\tn{for all }
	\ell \in \N,\ m \in \{0,1\}.
\end{equation}
For the global entropic regularization parameter in \eqref{eq:entropic-transport} we choose $\varepsilon=(\Delta x / 2)^2$, where $\Delta x = 2/N$ is the distance between adjacent pixels. With this regularization strength the entropic blurring is of the scale of the discretization, so we operate close to the unregularized regime. 
With larger $\veps$, the number of required domain decomposition iterations would decrease (since the freezing effect is less severe with more blur), as would the number of Sinkhorn iterations to solve each subproblem. On the other hand, the approximate support of $\pi_t$ (where mass is above the threshold of numerical precision) would increase and require more memory for storage.
The curl that has to be removed by the flow updates is by definition non-local and thus on a length scale above the entropic blur scale. The flow updates are therefore not affected as much by a change in $\veps$.

\begin{figure}[bt]
	\centering
	\includegraphics[width=\linewidth]{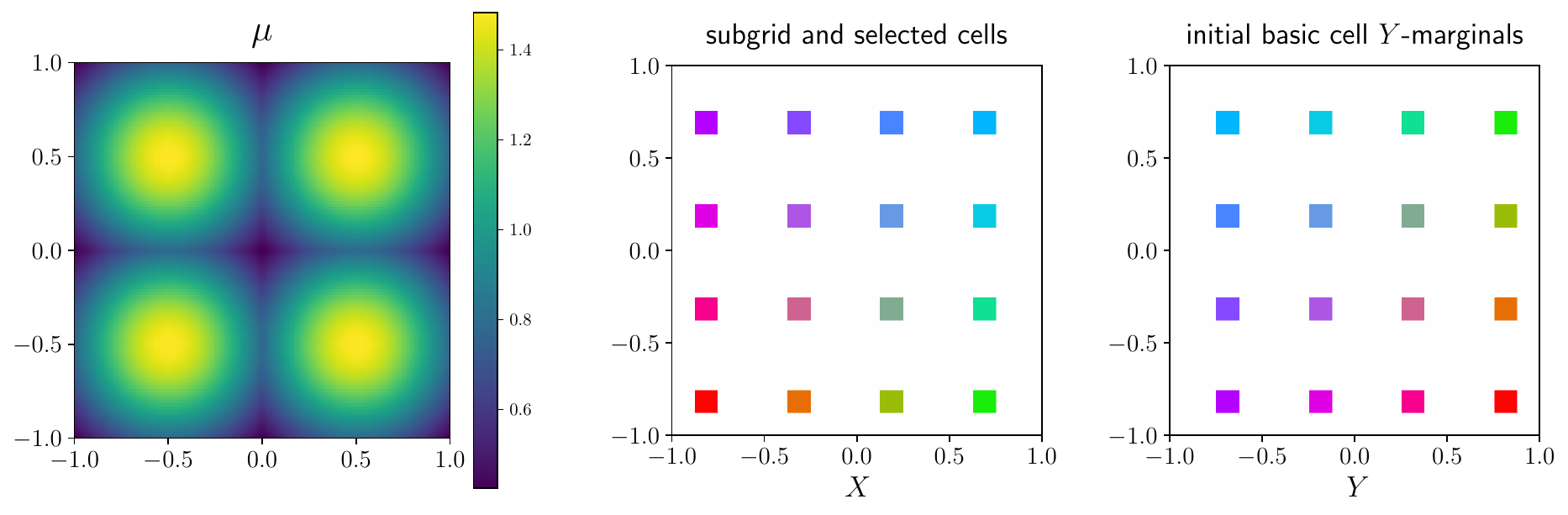}
	\caption{Left: Marginal $\mu=\nu$ for the experiments of Section \ref{sec:single-scale-experiments}. Center: Partition of $X$ into basic cells with $n = 16$ cells along each axis. Some basic cells are highlighted in color. Right: To visualize a transport plan $\pi \in \Pi(\mu,\nu)$, for each basic cell $X_i$ highlighted in color in the center panel, we show the $\nu$-density of the corresponding $Y$-marginal $\nu_i=\proj_Y( \pi \restr (X_i\times Y))$, cf.~\eqref{eq:basic-cell-Y-marginal}, in the same color on $Y$.
	Here this is exemplarily shown for the initialization $\pi^0=(\id,T_0)_\# \mu$ where $T_0$ is a rotation by angle $\pi/2$, \eqref{eq:T0}.
	}
	\label{fig:four-bumps}
\end{figure}

\paragraph{Visualization of trajectories.}
Visualizing the evolution of a coupling between 2-dimensional measures is challenging. 
We approach this issue as follows: We first color selected regions $A_i$ of $X$ as in Figure \ref{fig:four-bumps}, center.
Then for each $A_i$ we compute the density of the $Y$-marginal of $\pi \restr (A_i \times Y)$ with respect to $\nu$, i.e.~the function
$$\RadNik{\left(\proj_Y \pi\restr (A_i \times Y)\right)}{\nu}$$
which takes values in $[0,1]$ on $Y$ and we show this density in the same color as $A_i$. For the counterclockwise rotation $T_0$ and the induced plan $\pi_0$ this is done in Figure \ref{fig:four-bumps}, right. When the sets $A_i$ are (unions of) basic cells, this visualization can conveniently be combined with the domain decomposition algorithm.

\begin{figure}[bt]
	\centering
	\includegraphics[width=\linewidth]{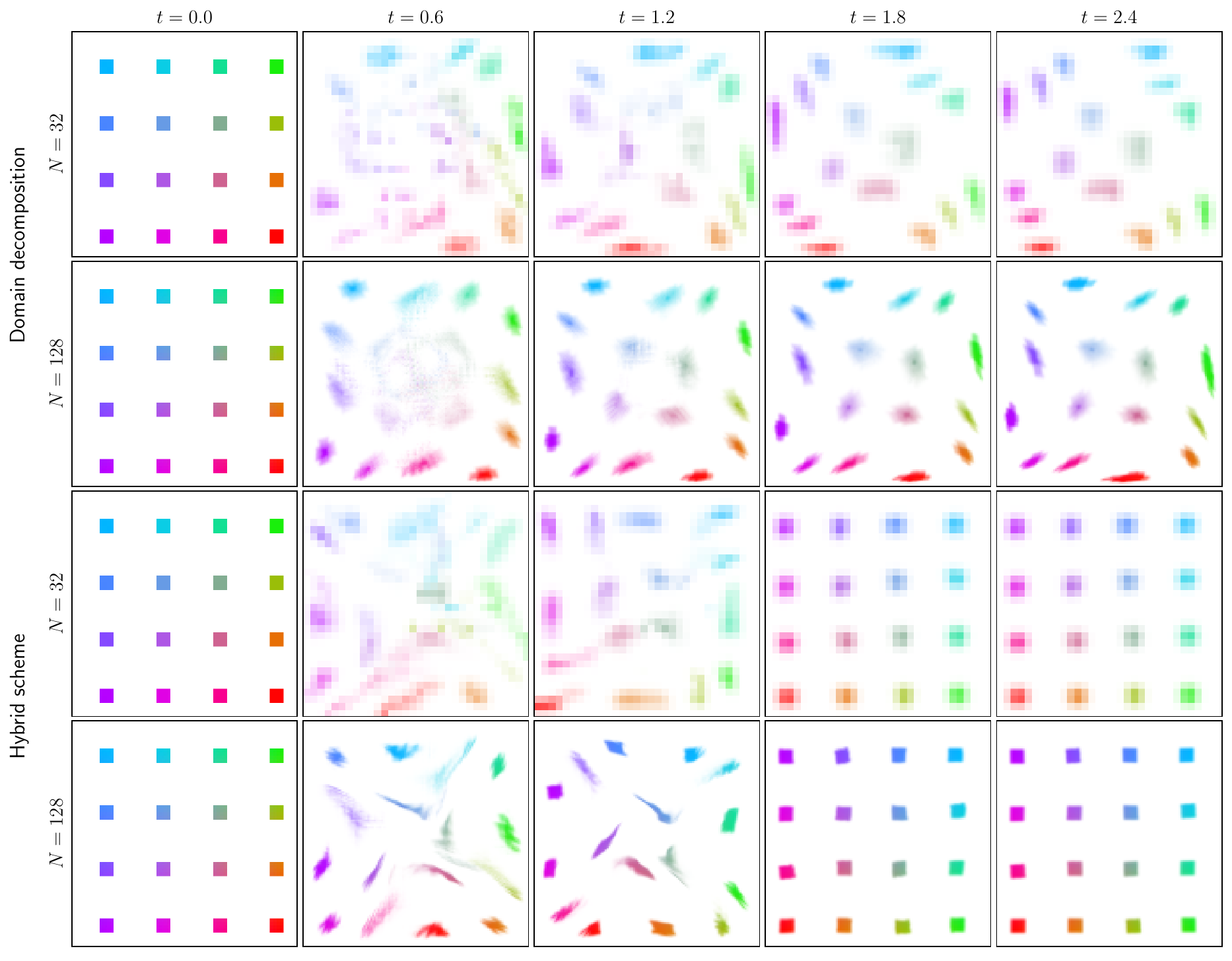}
	\caption{Domain decomposition and hybrid trajectories for $s = 2$, $\veps_\gamma=(\Delta x / 2)^2$. Domain decomposition alone clearly exhibits freezing. On the other hand, the hybrid scheme achieves n near-optimal configuration at approximately the speed for $N = 32$ and $N = 128$.}
	\label{fig:domdechybridcompare}
\end{figure}

\paragraph{The hybrid scheme resolves the freezing behavior consistently over discretization scales.}
The evolution of $\pi_0$ under the pure domain decomposition algorithm and the hybrid scheme is visualized in Figure \ref{fig:domdechybridcompare} for fixed $s=2$ and different image sizes $N \in \{32,128\}$, i.e.~for cell resolutions $n \in \{16,64\}$.
As expected, due to the presence of a global rotation, domain decomposition is already quite slow for $N=32$ and slows down further for $N=128$ (even when accounting for the time steps associated with each iteration). To appreciate this in detail observe the slow movement of the colored blobs close to the boundary. 

In contrast, the hybrid scheme evolves quickly towards the optimal coupling, with approximately equal speed for both $N \in \{32,128\}$, indicating that problems can consistently be solved also at fine resolutions.
Close to the optimal coupling the flow updates stop, in the sense that the optimal flows in \eqref{eq:MinCostFlow} are zero. This means that at the flow updates are no longer able to reach a better configuration and the residual optimization is done purely by the domain decomposition iterations.

\begin{figure}[bt]
	\centering
	\includegraphics[width=0.7\linewidth]{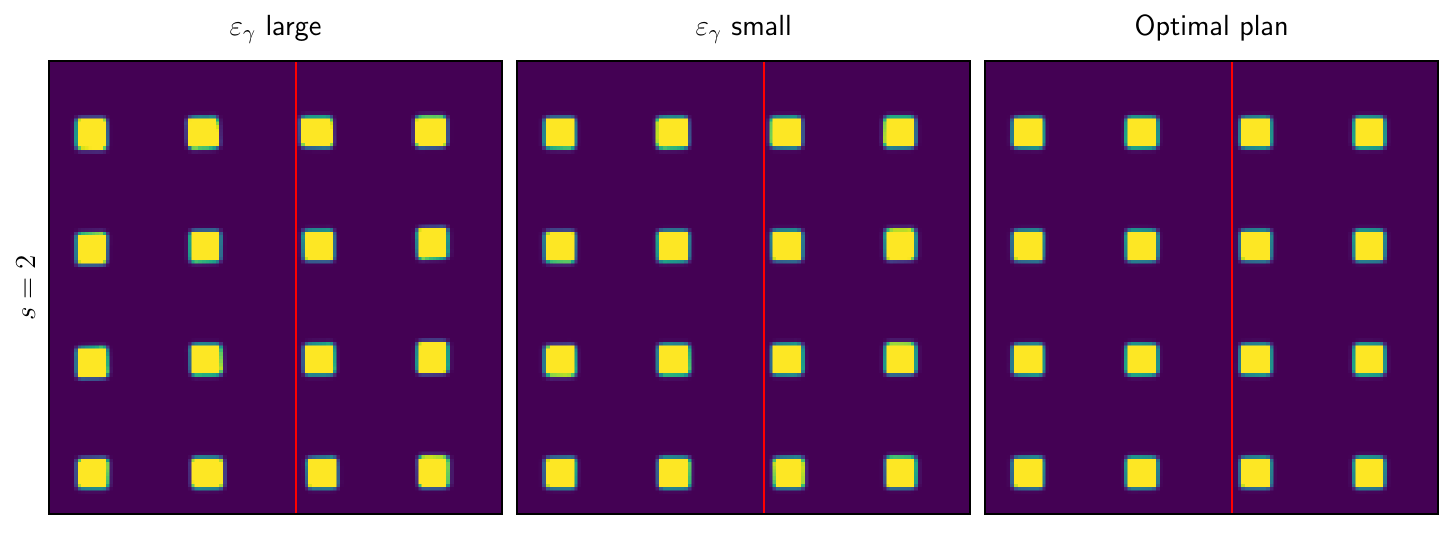}
	
	\vspace{-2mm}
	
	\includegraphics[width=0.7\linewidth]{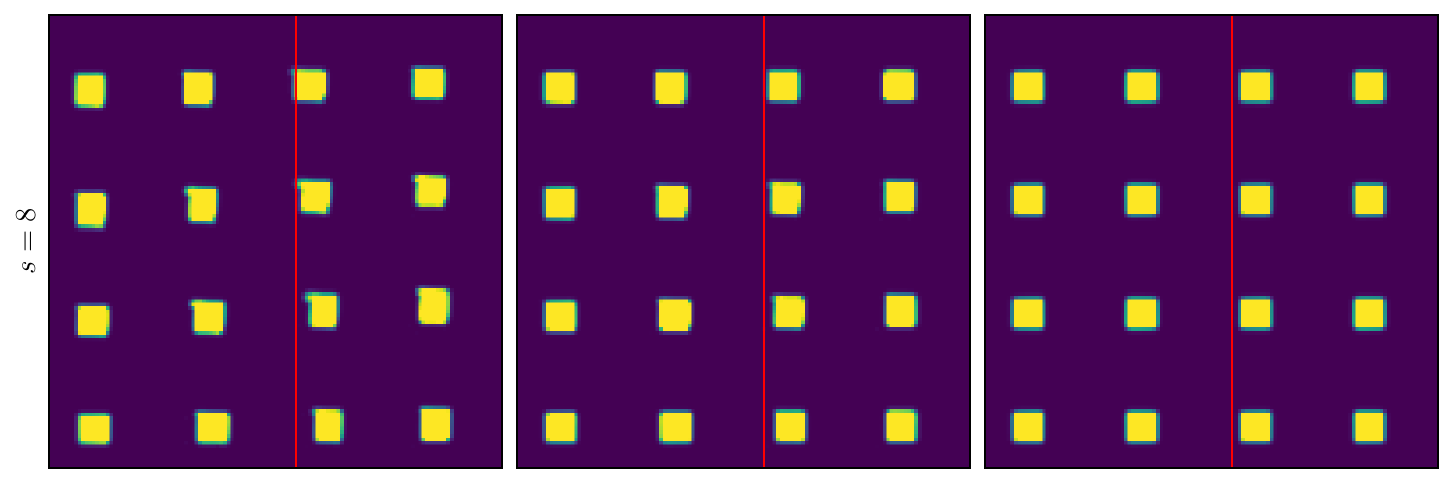}
	\caption{$Y$-marginals of selected basic cells at $t=4$, after the flow updates have stopped, for $N = 128$ and two different cellsizes $s$. The red line is drawn in the same location for all the pictures for reference. Note the difference in misalignment for $\veps_\gamma$ large and small (recall that $\veps_\gamma$ refers to the regularization strength used to compute the edge candidates, see Section \ref{sec:implementation-notes} for more details).} 
	\label{fig:comparison-eps-gamma}
\end{figure}

\paragraph{Influence of the edge candidates and relation to cellsize.}
The initial coupling $\pi_0$ features a strong global suboptimal rotation and thus we expect that edge candidates for all values of $\veps_\gamma$ will yield similar or even identical flow updates.
The intuitive reason is that, in this regime, the edge costs $(\cbar_{ij})_{ij}$ are dominated by the large rotation whereas the fine substructure within the basic cells has little influence.
Closer to optimality, on the other hand, the remaining curl is smaller and a more accurate estimation of the potential benefit of a flow update becomes beneficial.

Figure \ref{fig:comparison-eps-gamma} shows an example of residual rotation after the flow updates have stopped. 
We plot the hybrid trajectories for $N = 128$, with two different choices of cellsize ($s = 2$ and $s=8$) and for $\veps_\gamma$ being `small' ($(\Delta x/2)^2$) and `large' ($\infty$) at $t = 4$ (as we saw in Figure \ref{fig:domdechybridcompare} the flow updates stop approximately at $t = 2$). For comparison we show the optimal plan.
Our visualization choice for this plot consists in showing the aggregate cell $Y$-marginals for the same basic cells as in previous examples, so that one can focus on the displacement with respect to the optimal coupling.
The residual displacements are on the order of one or two pixels.

Cell size has the largest influence on the residual displacement. Smaller cell sizes allow for a finer rearrangement of mass, which is specially noticeable close to optimality. This must be weighed against an increased size and computational complexity of the associated min-cost flow problems. At large cell sizes, iterates with a smaller value of $\veps_\gamma$ reach a smaller residual displacement, as the cost coefficients in the flow update can capture the benefit of a rearrangement more accurately.

To sum up, the hybrid scheme iterates converge fast towards the global optimum for a range of initializations and resolutions. 
When flow updates stop to make progress, there remains a residual displacement on the scale of one or two pixels, that depends on the cell size and the level of resolution used in the edge candidates.

\subsection{Flow updates and multiscale domain decomposition}
In \cite{BoSch2020} a multiscale version of domain decomposition was presented and it is straight-forward to combine this with flow updates.
However we observed that in this combined scheme flow updates were almost always zero.
The rare non-zero updates were usually local and would also have been obtained by the next domain decomposition iteration.
 
While this observation is somewhat disappointing from the perspective of the flow updates, it once more underlines the efficiency of the multiscale scheme, as observed in \cite{BoSch2020} where only a fixed number of multiscale iterations was required for solving large problems with high accuracy, where the number of iterations only increases logarithmically with the image size.
Freezing does not seem to be an issue in combination with the multiscale scheme.
Intuitively it appears plausible that the multiscale scheme does not introduce substantial curl as it moves from coarse to fine scales.
A theoretical confirmation of this fact would be an interesting and challenging problem for future work.

\subsection{Quantitative analysis and comparison with multiscale domain decomposition}
\label{sec:quantitative-comparison}
\begin{figure}[bt]
	\centering
	\includegraphics[width=0.7\linewidth]{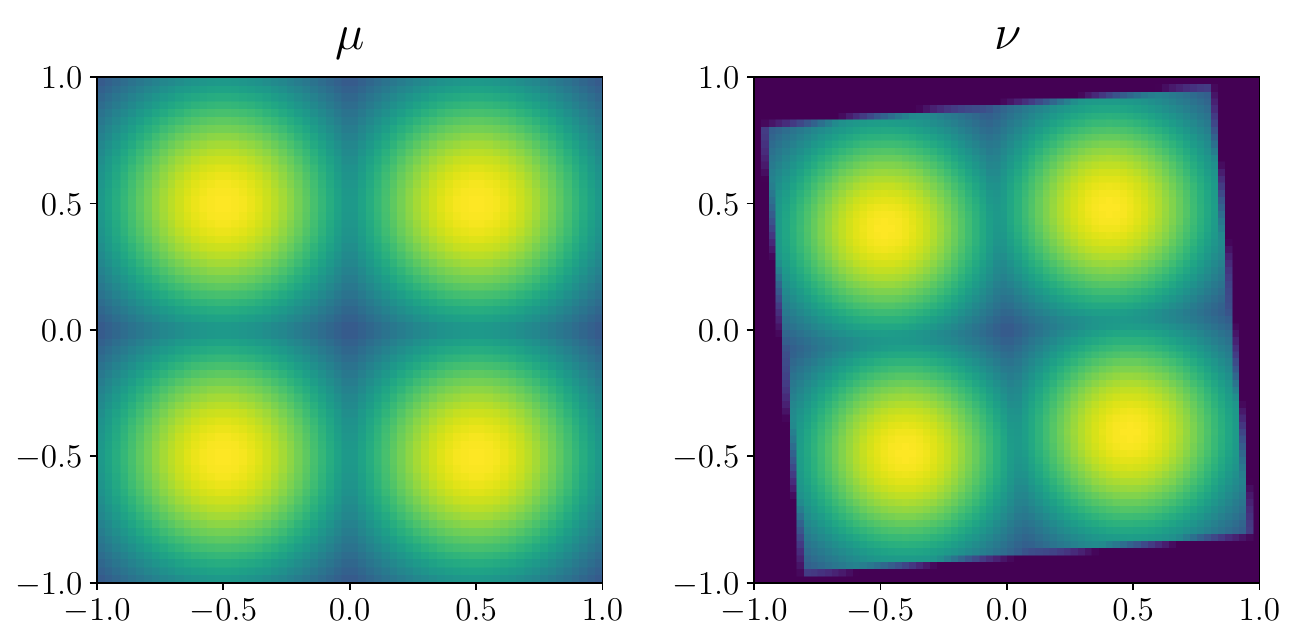}
	\caption{Initialization for the experiments in Section \ref{sec:quantitative-comparison}. The $Y$-marginal is obtained by applying a rotation with angle $\theta$ (above, $\theta = 5^\circ$) to the marginal $\mu$. This rotation provides us with an initialization of the transport plan for the single-scale experiments.}
	\label{fig:multiscale-init}
\end{figure}

In Section \ref{sec:single-scale-experiments} a qualitative comparison between basic domain decomposition and the hybrid scheme was given.
It was demonstrated that domain decomposition becomes arbitrarily slow on large problems where the initialization contains curl and that the hybrid scheme resolves this issue consistently.
In this Section we perform a quantitative comparison between multiscale domain decomposition, single-scale domain decomposition, and the hybrid scheme.
We focus on a modified version of the experiment proposed in Figure \ref{fig:four-bumps} where we obtain the $Y$-marginal $\nu$ as a rotation of the $X$ marginal $\mu$ by an angle $\theta$ and a contraction, such that the rotated image fits into $[-1,1]^2$. This transformation map also serves as non-optimal initialization for the transport plan in the single-scale solvers (domain decomposition and hybrid). An example is given in Figure \ref{fig:multiscale-init}. 
We monitor the primal score of the algorithms, compute the relative primal-dual gap (using for the dual score the one obtained from multiscale domain decomposition) and plot its evolution with respect to both the iteration number and the running time.
The results are summarized in Figure \ref{fig:chase-four-bumps}.

As before, basic domain decomposition at a single scale decreases the score slowly, the effect becoming more pronounced as $N$ increases (while keeping $s$ fixed).
The hybrid scheme decreases the score much more quickly in the fist stage of the algorithm. We then observe the stop of the flow updates and a slow decrease in score as the domain decomposition iterations remove the residual deformations at a similar slope as pure domain decomposition.
The multiscale scheme is clearly the strongest algorithm in this comparison. It consistently reaches the smallest PD gap in a short time.
This is particularly true for $\theta=20^\circ$.

For $\theta=5^\circ$ the initialization $\pi^0$ is presumably relatively close to being optimal. In this case we see that the flow updates in the hybrid scheme stop at a time that is comparable with the runtime of the multiscale scheme.
The flow updates stop at a higher score, but the relative difference in these scores decreases as the resolution increases.
This suggests that flow updates could be a relevant algorithmic alternative when a good approximate initial coupling is known and thus improving this initialization could be preferable to re-solving from scratch.
However, for this to become fully viable it seems that additional work is required to make flow updates more efficient and more accurate.

\begin{figure}
	\includegraphics[height=0.62\linewidth]{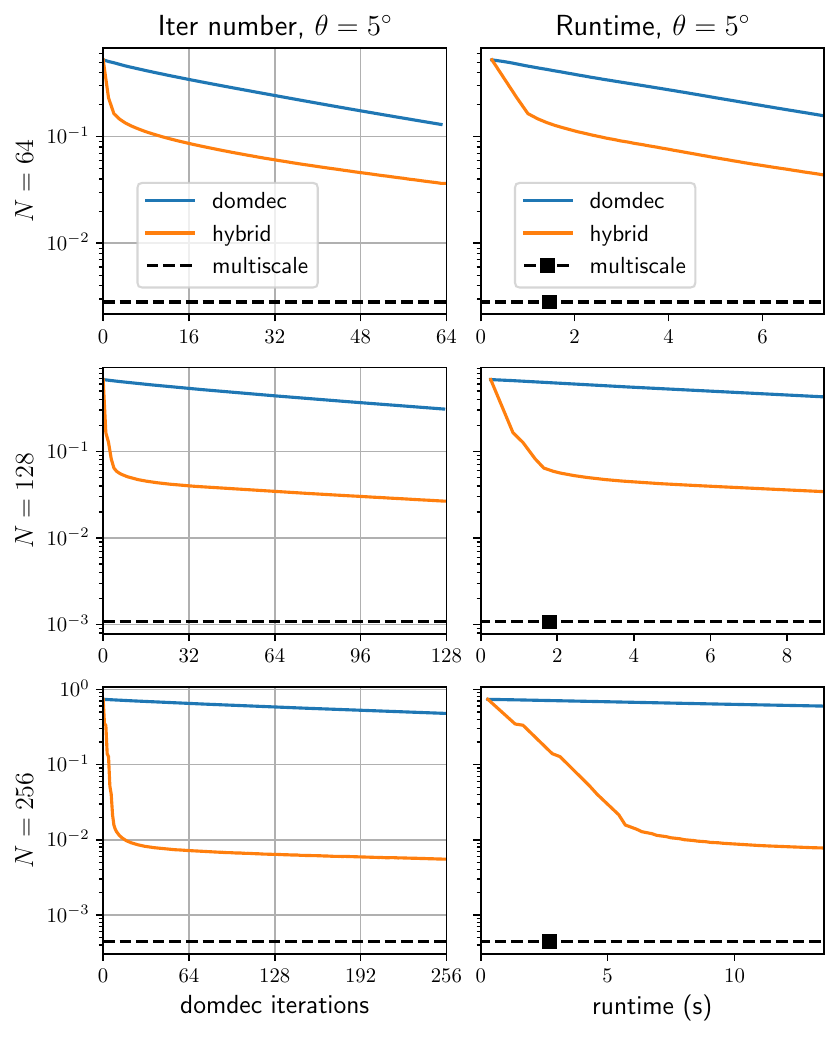}
	\includegraphics[height=0.62\linewidth]{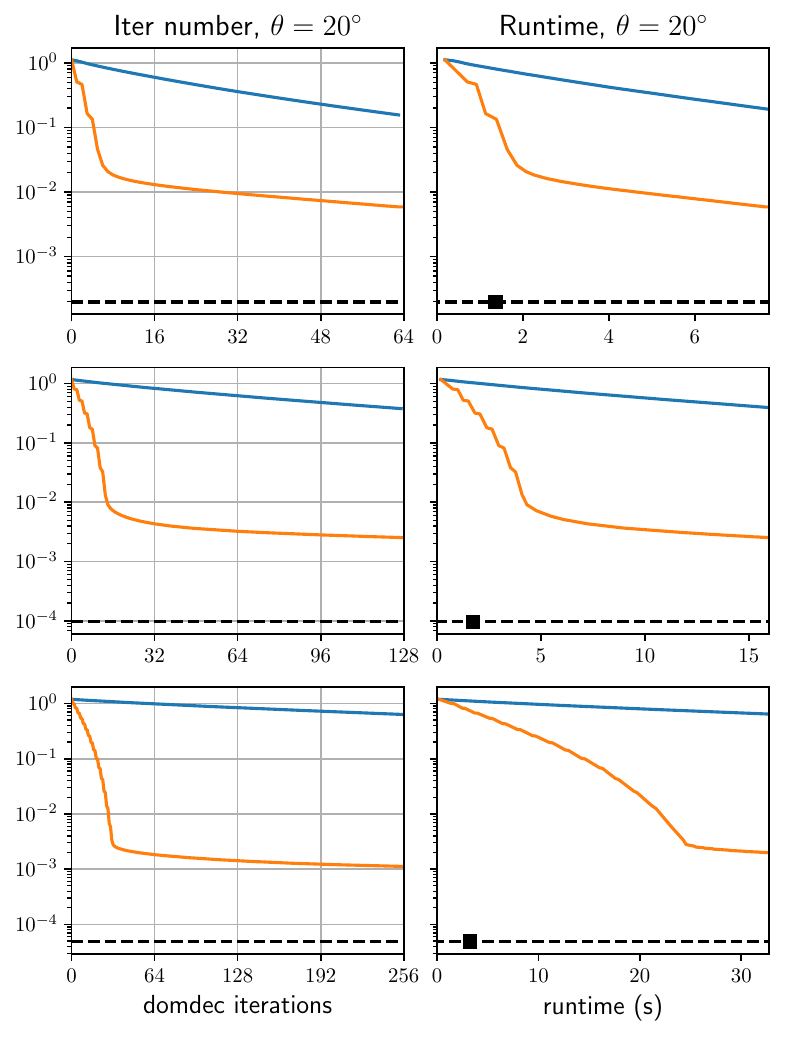}
	\caption{Relative primal-dual gap as a function of the iteration count and runtime, for $\theta \in \{ 5^\circ, 20^\circ\}$ (see Figure \ref{fig:multiscale-init}), $N \in \{64,128,256\}$ and $s=2$. The black square in the runtime column indicates the runtime of the multiscale scheme.}
	\label{fig:chase-four-bumps}
\end{figure}

\section{Conclusion}

In this article we introduced flow updates as a remedy for the freezing of singlescale domain decomposition for entropic optimal transport.
These updates can be interpreted as an $\LL^\infty$-variant of the famous Angenent--Haker--Tannenbaum scheme.
This specific variant has the advantage that it can be discretized as a min-cost flow problem that exactly preserves the marginals of a discrete transport plan. It is always well-posed (provided the set of basic cells $I$ is finite) and can also operate on transport plans that are not supported on the graph of a map, e.g.~as in entropic optimal transport.
In addition it can be combined in a natural way with domain decomposition, leading to an algorithm that converges to the globally optimal solution.

One limitation is that a global min-cost flow problem on the set of basic cells needs to be solved at each update, implying a trade-off between the size of this problem and its accuracy in predicting favourable flows.
We find that far from the optimal solution, flow updates work well, even on coarse grids, whereas closer to the optimal solution some residual suboptimality remains that cannot be resolved by flow updates.

We observe that the hybrid scheme works much faster than single scale domain decomposition and consistently resolves the issue of freezing.
We also observe that the multiscale domain decomposition avoids the issue of freezing and in general is faster than the hybrid scheme.
The latter could be a viable alternative in cases where a good approximate initial coupling is available.
  
Methods that solve the optimal transport problem by gradually improving a given initial coupling are conceptually appealing and it would be interesting to study whether flow updates can be improved to become truly competitive with multiscale domain decomposition.
Studying variants of the AHT scheme that are well-posed gradient flows on the space of transport plans and reliably reach the globally optimal solution also seems an interesting theoretical question in its own right.
Other open questions are whether it  can be established theoretically that multiscale domain decomposition avoids freezing,  or whether domain decomposition can be applied to multi-marginal transport.
 
\section*{Acknowledgements}
IM and BS were supported by the Emmy Noether programme of the DFG.
We thank the anonymous referees for their valuable suggestions and comments.

\bibliography{references}{}
\bibliographystyle{plain}

\newpage

\appendix
\section{A GPU implementation of domain decomposition for entropic optimal transport}
\label{sec:gpu}
An important factor in the adoption of entropic optimal transport has been the availability of fast implementations of the Sinkhorn algorithm on GPUs such as \texttt{keops} and \texttt{geomloss} \cite{feydy_keops,feydy_geomloss}.
It is thus natural to consider using these solvers for the cell subproblems in domain decomposition. However, in each domain decomposition iteration one must solve many small problems with irregular shapes on different domains, which in principle is not a GPU-friendly task.

Section \ref{sec:gpu-domdec} describes the necessary adaptations to make domain decomposition amenable to GPU computing when the marginals $\mu$ and $\nu$ live on low-dimensional, regular grids. Section \ref{sec:multiscale-experiments} compares its performance against CPU-based parallel domain decomposition, and against a performant GPU implementation of the global Sinkhorn algorithm.

\subsection{Implementation details}
\label{sec:gpu-domdec}

The main obstacles for a performant GPU implementation of domain decomposition are the following:

\begin{itemize}
	\item The subproblems in domain decomposition are numerous, small, and with diverse shapes. In contrast, state-of-the-art implementations of the Sinkhorn algorithm, such as \cite{feydy_keops,feydy_geomloss}, perform best when working with monolithic tensors of large size.
	\item In previous implementations of domain decomposition, the current cell $Y$-marginals $(\nu_i)_{i\in I}$ are stored in a sparse structure. However, GPUs stream their operations best on contiguous data. 
\end{itemize}
These issues are addressed by the following adaptations.

\paragraph{Bounding box representation.} The basic cell $Y$-marginals $(\nu_i)_{i\in I}$ are now stored in a \textit{bounding box} structure. This consists of an array \texttt{data} of shape $(|I|, s_1, ..., s_d)$ and an array \texttt{offsets} of shape $(|I|, d)$. For each basic cell $i \in I$ the slice \texttt{data[$i$]} contains a translated, truncated version of $\nu_i$ on a small regular grid of size $s_1 \times \ldots \times s_d$, and the offset of this translation is saved in \texttt{offsets[$i$]}. This translation is chosen such that the dimensions of the bounding box $(s_1, ..., s_d)$ are minimal, while containing all points of $\nu_i$ above a truncation threshold. Figure \ref{fig:boundingbox} exemplifies this bounding box representation. We refer to the set of basic cell $Y$-marginals in bounding box structure as $\nu_I$.

\begin{figure}[hbtp]
	\centering
	\includegraphics[width=\linewidth]{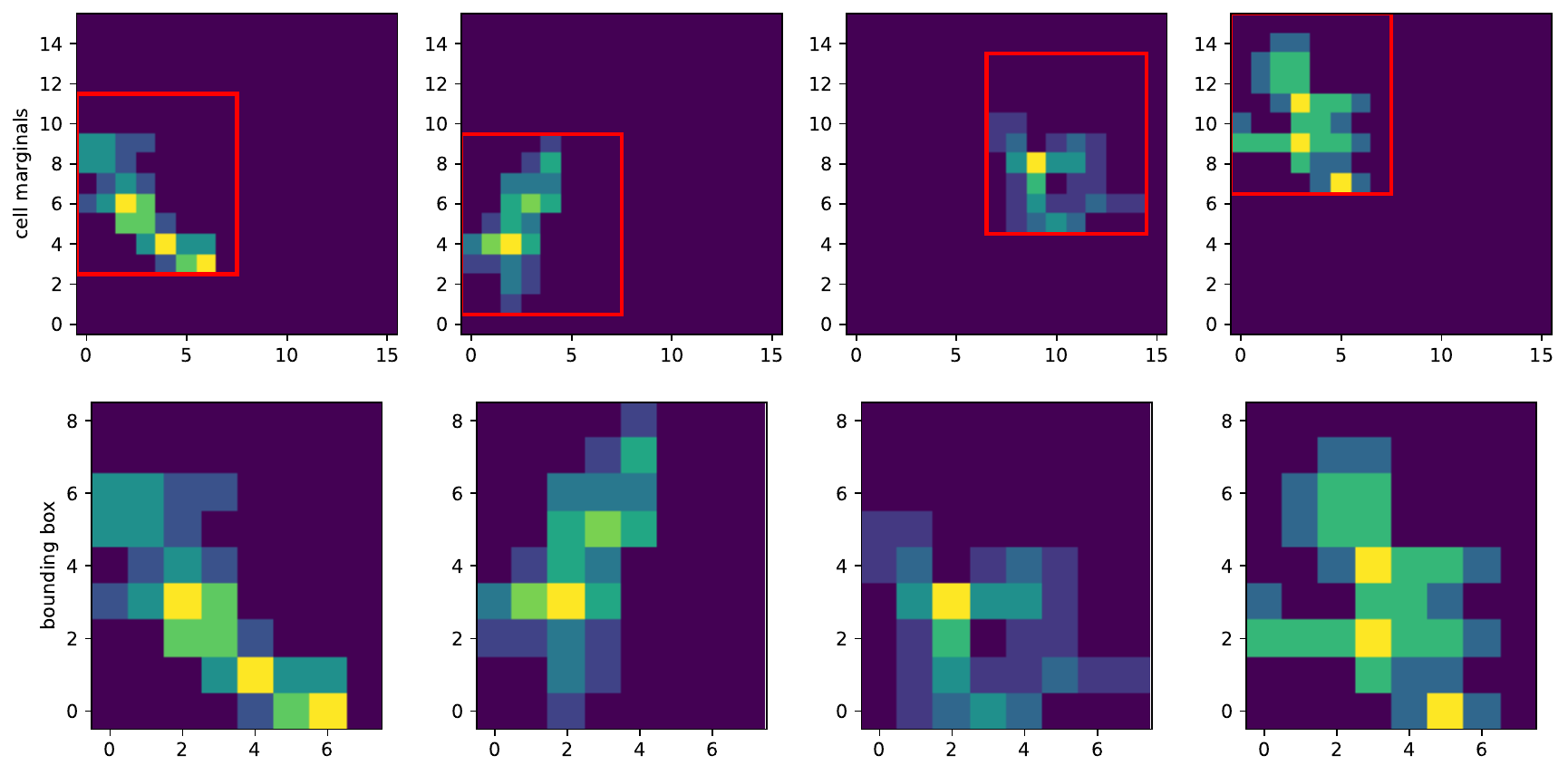}
	\caption{Top row: cell marginals $(\nu_i)_{i \in I}$ for $I = \{0,1,2,3\}$ and $Y = \{0,...,15\}^2$. The maximum width and height of their supports is computed, and a red bounding box with these dimensions is pictured. 
		Bottom row: using these maximum dimensions we can store the marginals in a compressed and contiguous fashion. 
		The original coordinates of the bottom-left corner of each marginal are saved in the tensor \texttt{offsets}.}
	\label{fig:boundingbox}
\end{figure}

\paragraph{Basic cell manipulation.} For solving the cell problems in partition $\partGeneric$  one must aggregate basic cell marginals into composite cell marginals as follows:
\begin{equation}
\label{eq:get-composite-cell-marginal}
\nu_J \assign \sum_{i\in J} \nu_i
\qquad
\tn{ for each $J$ in $\partGeneric$.}
\end{equation}
For $(\nu_i)_{i\in I}$ given in the bounding box representation $\nu_I$, one must take the original offsets and widths of each basic cell marginal into account to compute the corresponding composite cell marginals.	We implement this operation in a function \textsc{BasicToComposite}, which returns the bounding box representation of the composite cell marginals, denoted by $\nu_{\partGeneric}$. In order to obtain a performant implementation we devised a custom CUDA kernel for this task.

The composite cell $X$-marginals stay constant for any given partition $\partGeneric$, and can also be cast into a bounding box structure $\mu_\partGeneric$.

\paragraph{Custom CUDA Sinkhorn.} To achieve the best performance for our particular type of problems, we implement the Sinkhorn reduction directly in CUDA. Our implementation has the following features: 
\begin{itemize}
	\item It is a batched implementation, i.e.~several problems can be solved in parallel in the same kernel call.
	As input it takes a set of composite $X$- and $Y$-cell marginals $\mu_\partGeneric$ and $\nu_\partGeneric$ (in bounding box format), as well as an initialization for the dual potentials on the $X$-side, $\alpha^{\rm init}_\partGeneric$, and returns the optimal dual potentials $\alpha_\partGeneric$ and $\beta_\partGeneric$.
	It can however also be used to solve a single problem.
	\item Following \cite{feydy_keops,feydy_geomloss}, the Sinkhorn iterations are performed in the logarithmic domain, to avoid issues with numerical instability.
	\item Again following \cite{feydy_keops, feydy_geomloss}, it is \textit{online}, that is, it computes the transport cost matrix on the fly and does not store it. This allows us to keep the memory use linear in the problem size. 
	\item For the squared cost we provide a \textit{separable} implementation of the Sinkhorn loop, as proposed in \cite{Solomon-siggraph-2015}, which reduces its computational complexity.
	\item We design it to be particularly performant at solving many small problems in parallel. Since in our case every individual marginal is relatively small (especially in the separable implementation), one can load each dual variable to shared memory and compute each whole line of the Sinkhorn's logsumexp reduction in a single CUDA thread. Since each given cell dual potential is used for multiple Sinkhorn reductions (corresponding to computing each entry of the conjugate dual potential), the GPU's shared memory is used efficiently and communication time is reduced significantly.
	This contrasts with usual strategies for computing reductions on GPUs \cite{reduction_CUDA}, which are optimized for large tensors sizes and typically employ several threads processing synchronously different parts of the data, with a subsequent additional synchronization step.
\end{itemize}
In the following we refer to this implementation as \textsc{SinkhornGPU}.

\paragraph{Online reduction for basic cell marginals.} Upon convergence of the subproblems on the composite cells, one must compute the new basic cell $Y$-marginals $\nu_I$, see \eqref{eq:basic-cell-Y-marginal}. 
A naive way to obtain these would be to first instantiate the optimal composite cell plans $\pi_\partGeneric$ and then perform a reduction on basic cells. However, this would be very memory-intensive. Instead, the partial marginals can be computed directly via a partial logsumexp reduction, which corresponds to a Sinkhorn half-iteration on the composite cells, restricted to the basic cells, as follows (for simplicity stated on a discrete problem):
\begin{align*}
\nu_i(\{y\})
&=
\sum_{x\in X_i}
\pi_J(\{(x,y)\})
=
\sum_{x\in X_i}
\exp \left(
\frac{\alpha_J(x) + \beta_J(y) - c(x,y)}{\veps}
\right)		
\mu(\{x\})\nu(\{y\})
\nonumber
\\
&=
\exp \left(
\frac{\beta_J(y)}{\veps}
\right)	
\nu(\{y\})
\sum_{x\in X_i}
\exp \left(
\frac{\alpha_J(x) - c(x,y)}{\veps}
\right)		
\mu(\{x\})
\nonumber
\\
&=
\exp \left(
\frac{\beta_J(y)-\tilde{\beta}_{J,i}(y)}{\veps}
\right)	
\nu(\{y\})
,
\\
\tn{where }
\tilde{\beta}_{J,i}(y)
&\assign
- \veps \log \sum_{x\in X_i}
\exp \left(
\frac{\alpha_J(x) - c(x,y)}{\veps}
\right)		
\mu(\{x\}).
\end{align*}
This can be done efficiently thanks to the enhancements outlined above. 
We refer to the function that computes the new basic $Y$-marginals from the composite ones and the entropic potentials as \textsc{CompositeToBasic}.

\paragraph{Batching and clustering.}
While the bounding box format for the cell marginals provides a drastic reduction in memory compared to a dense format, it is not as memory efficient as more sophisticated (but less GPU-affine) sparse structures. This effect becomes more severe if a very large number of marginals has to be represented by a common bounding box size: a single large problem can increase the memory demand for all other problems in an unnecessary way. For these two reasons we solve the composite cell problems in batches, and these batches are \textit{clustered} according to bounding box size. This means that fewer composite cell problems need to be held in memory at any given time, and memory bandwidth is used more efficiently due to more appropriate bounding box sizes.
An illustration of bounding box clustering and batching is shown in Figure \ref{fig:clustering} and a brief discussion of its efficiency is given in Section \ref{sec:multiscale-experiments}.
Without this batching we would run of of memory for the largest problems considered in our experiments.

We will refer to the function that constructs the batches as \textsc{CreateBatches}. It divides the current partition $\partGeneric$ into a set of subpartitions given by $(\partBatch_b)_{b=1}^{B}$. Note that every batch $\partBatch_b$ of composite cells is associated to a batch of basic cells given by  $\indexBatch_b \assign \cup_{J \in \partBatch_b} J$.
In addition, we implement a function \textsc{CombineBatches} that combines the solutions of the separate batches.

\paragraph{Other components.} The features outlined above are specific to adapting domain decomposition to GPU computing. There are other implementation details that are necessary for domain decomposition to be efficient on any architecture, such as coarse-to-fine solving, epsilon scaling, truncation and balancing. All these are described in \cite[Section 6]{BoSch2020} and can be adapted to the GPU without difficulties, since most steps are parallelizable operations on the basic cell marginals. 

Algorithm \ref{alg:DomDecIterGPU} summarizes a practical implementation of the domain decomposition iteration on the GPU. We refer to this as \DomDecGPU.

\begin{algorithmfloat}[hbtp]
	\newcommand{\myitem}{\textbullet\,}
	\noindent
	\textbf{Input}:
	
	\myitem feasible basic $Y$-marginals $\nu_I$, in bounding box format.\\
	\myitem partition $\partGeneric$\\
	\myitem initialization for scaling factors $\alpha^{\rm init}_\partGeneric$
	
	\noindent
	\textbf{Output}: new $\nu_I$ and $\alpha_\partGeneric$.
	\smallskip
	
	\begin{algorithmic}[1]
		\ForAll{$\partBatch \in \textsc{CreateBatches}(\partGeneric)$}
		\State $\tilde{I} \leftarrow \cup_{J\in \partBatch}\ J$ 
		\State $\nu_{\partBatch} \leftarrow \textsc{BasicToComposite}(\nu_I, \partBatch) $
		\Comment{combine basic into composite cells}
		\State $\alpha_{\partBatch}, \beta_{\partBatch} \leftarrow \textsc{SinkhornGPU}( \mu_{\partBatch}, \nu_{\partBatch}, \alpha^{\rm init}_{\partBatch})$ 
		\Comment{solve entropic problems with Sinkhorn}
		\State $\nu_{\indexBatch} \leftarrow \textsc{CompositeToBasic}(\mu_{\indexBatch}, \nu_{\partBatch}, \alpha_{\partBatch}, \beta_{\partBatch})$
		\Comment{get basic cell $Y$-marginals}
		\State $\textsc{BalanceGPU}(\mu_{\indexBatch},\nu_{\indexBatch}, \partBatch)$
		\Comment{balance measures in each composite cell}
		\State $\textsc{TruncateGPU}(\nu_{\indexBatch})$
		\Comment{remove negligible mass points}
		\EndFor
		\State $\nu_I \leftarrow \textsc{CombineBatches}(\nu_{\indexBatch_1},...,\nu_{\indexBatch_B})$
		\State $\alpha_\partGeneric \leftarrow \textsc{CombineBatches}(\alpha_{\partBatch_1},...,\alpha_{\partBatch_B})$
	\end{algorithmic}
	\caption{Practical implementation of Algorithm \ref{alg:DomDecIter} on the GPU}
	\label{alg:DomDecIterGPU}
\end{algorithmfloat}

\subsection{Numerical experiments}
\label{sec:multiscale-experiments}

\subsubsection{Setup}
\label{sec:preliminaries-numerics}

\paragraph{Compared algorithms and implementations.} In this Section we compare our GPU domain decomposition (\DomDecGPU) to two related algorithms: parallel domain decomposition on the CPU with parallelization via MPI (\DomDecMPI, using the implementation of \cite{BoSch2020}) and the Sinkhorn algorithm on the GPU without domain decomposition (\SinkhornGPU).
The former uses a dense matrix-multiplication Sinkhorn with absorption and automatic $\veps$-fixing, see \cite[Section 7.4]{BoSch2020} for more details. For the latter we use the implementation \SinkhornGPU\ discussed above, which compares favorably to an implementation based on \texttt{keops/geomloss} in all the range of problem sizes considered (see Figure \ref{fig:comparison-sinkhorn-solvers}).

\paragraph{Hardware.} We ran our experiments at GWDG's\footnote{\url{https://gwdg.de/}} Scientific Compute Cluster, which uses a Linux operating system.
We run \DomDecMPI\ on a single compute node with an Intel Xeon Platinum 9242 Processor containing 48 cores (we use 1 master core and 16 working cores). 
For the GPU algorithms we use a node equipped with an Intel Xeon Gold 6252 Processor with 24 cores (we use 1) and an NVIDIA V100 with 32 GB of memory.

\begin{figure}[hbt]
	\centering	\includegraphics[width=0.5\linewidth]{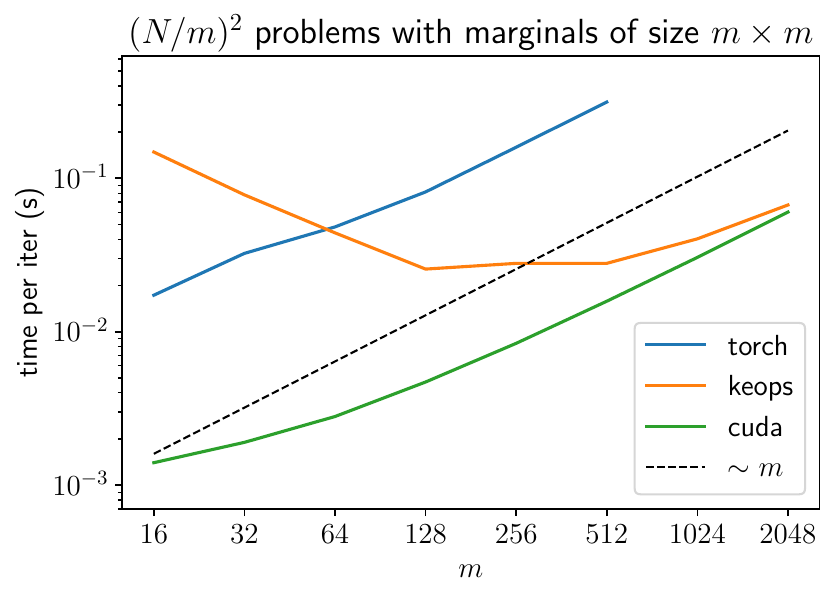}
	\caption{Comparison of Sinkhorn GPU implementations with different backends: Time per Sinkhorn iteration on a batch of problems of the same size. All solvers use log-domain and separable kernels. The \texttt{torch} solver is dense, \texttt{keops} uses the homonymous library, and \texttt{cuda} stands for our custom CUDA implementation. To observe dependence both in batch dimension and in problem size, we solve simultaneously $(N/m)^2$ problems with size $m^2$, for $N = 2048$. The theoretical complexity grows linearly in $m$, which is captured by the \texttt{torch} solver (with a large, constant overhead). The \texttt{keops} solver shows additional overhead for small problem sizes but performs faster on large problems. Our implementation is efficient on large batches of small problems, without compromising its efficiency for large problems.}
	\label{fig:comparison-sinkhorn-solvers}
\end{figure}

\paragraph{Stopping criterion and measure truncation.} 
We use the $L^1$ marginal error of the $X$-marginal after the $Y$-iteration as stopping criterion for the Sinkhorn iterations. This error is required to be smaller than $||\mu_\partGeneric|| \cdot \tn{Err}$, where $\mu_\partGeneric$ gathers the $X$-marginals of the problems the Sinkhorn solver is handling: an individual problem for \DomDecMPI, the global problem for \SinkhornGPU, or a collection of problems (possibly all of them) for \DomDecGPU. We use $\tn{Err} = 10^{-4}$. For comparison, for \SinkhornGPU\ we test in addition the value $\tn{Err} = 10^{-3}$.
Partial marginals $(\nu_i)_{i\in I}$ are truncated at $10^{-15}$ and stored as sparse vectors in \DomDecMPI, or in bounding box structures in \DomDecGPU.

\paragraph{Test data.} We use the same problem data as in \cite{BoSch2020}: Wasserstein-2 optimal transport on 2D images with dimensions $2^n \times 2^n$, with $n$ ranging from $n=6$ to $n = 11$, i.e.~images between the sizes $64 \times 64$ to $2048 \times 2048$. The images are Gaussian mixtures with random variances, means and magnitudes. This represents challenging problem data, featuring strong compression, expansion and anisotropies. For each problem size we generated 10 test images, i.e.~45 pairwise non-trivial transport problems, and average the results. Examples of the images are shown in Figure \ref{fig:ExampleImages}.

\begin{figure}[hbt]
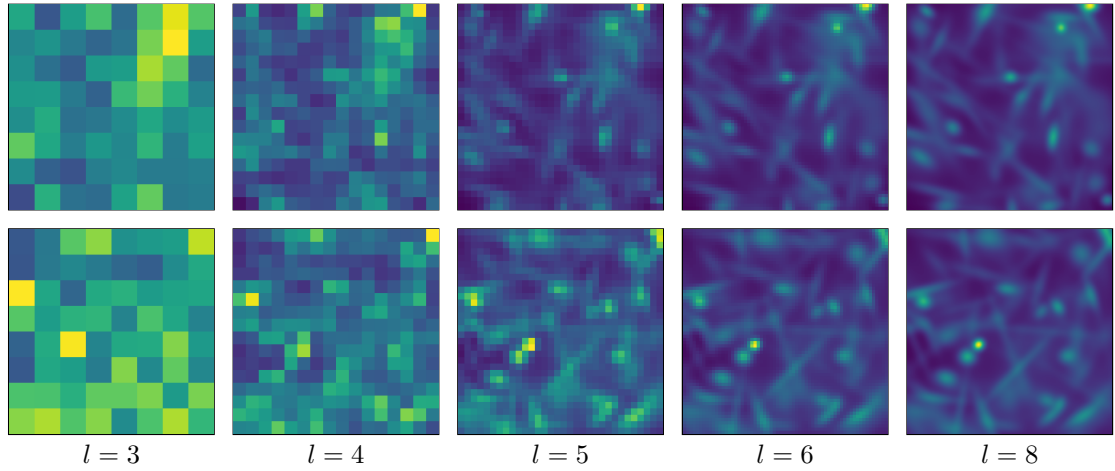

	\centering
	{\def\imgw{2.7cm}
		\begin{tikzpicture}[x=\imgw,y=\imgw,img/.style={inner sep=0pt,draw=black,line width=1pt,anchor=north west}]
		\foreach \x/\y/\l in {0/0/3,1/0/4,2/0/5,3/0/6,4/0/8} {
			\node[img] at (1.1*\x,-1.1*\y){\includegraphics[width=\imgw]{fig/atomic_cells/muX_l\l.png}};
		}
		\begin{scope}[shift={(0,-1.1)}]
		\foreach \x/\y/\l in {0/0/3,1/0/4,2/0/5,3/0/6,4/0/8} {
			\node[img] at (1.1*\x,-1.1*\y)[label=below:{$l=\l$}]{\includegraphics[width=\imgw]{fig/atomic_cells/muY_l\l.png}};
		}
		\end{scope}
		\end{tikzpicture}
	}
	\caption{Two example images for size $256 \times 256$, at different layers, $l=8$ being the original. }
	\label{fig:ExampleImages}
\end{figure}

\paragraph{Multi-scale and $\veps$-scaling.} All algorithms implement the same strategy for multi-scale and $\veps$-scaling as in \cite[Section 6.4]{BoSch2020}. At every multiscale layer, the initial regularization strength is $\varepsilon = 2(\Delta x)^2$, where $\Delta x$ stands for the pixel size, and the final value is $\varepsilon = (\Delta x)^2 / 2$. In this way, the final regularization strength on a given multiscale layer coincides with the initial one in the next layer. Finally, on the finest layer $\varepsilon$ is decreased further to the value $(\Delta x)^2 / 4$ implying a relatively small residual entropic blur.

\paragraph{Double precision.}
All algorithms use double floating-point precision. 
For single precision we observed a degradation in the accuracy for problems with $N \gtrsim 256$. 
The cause for this behavior is that, for the chosen objective error $\rm Err$ and regularization strengths $\veps$, the update to the dual potentials according to the Sinkhorn iterations eventually has a relative magnitude below the machine precision of single floating-point precision ($2^{-24}$). Consequently, the Sinkhorn iterations using single precision may stall before achieving the required solution quality.

\begin{figure}[hbt]
	\includegraphics[width=\textwidth]{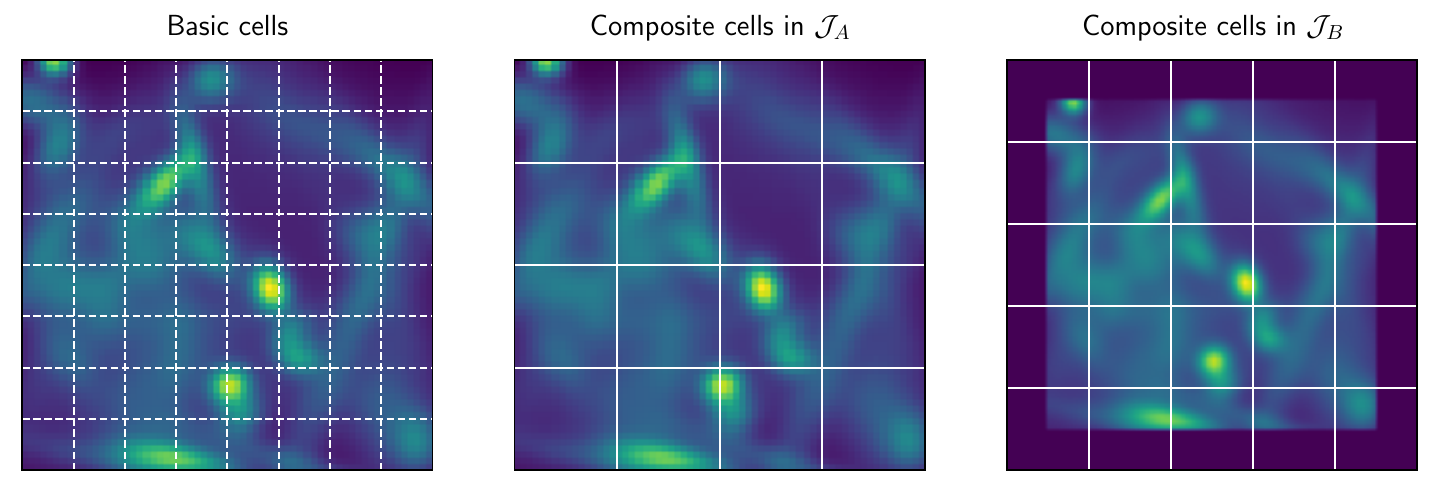}
	\caption{Basic and composite cells for $N = 64$, $s = 8$. For $B$ cells we pad the image with a margin of width $s$.}
	\label{fig:basic-and-composite}
\end{figure}

\paragraph{Basic and composite partitions.} 
We generate basic and composite partitions as in \cite{BoSch2020}: for the basic partition we divide each image into blocks of $s\times s$ pixels (where $s$ is a divisor of the image size), while composite cells are obtained by grouping cells of the same size, as shown in Figure \ref{fig:basic-and-composite}.
In the GPU version, for the $B$ cells we pad the image boundary with a region of very low constant mass, such that all composite cells have the same size, to simplify compatibility.
The choice of $s$ entails a performance trade-off. Larger $s$ implies fewer problems and fewer domain decomposition iterations until convergence, but increased complexity in the individual sub-problems.
As in \cite{BoSch2020}, the choice of cell size $s=4$ yields the best results.

\paragraph{Batching and clustering} As anticipated in Section \ref{sec:gpu-domdec}, solving all composite cell problems simultaneously becomes challenging for big resolutions; in fact without further adjustments we run out of memory for $N = 2048$.
Therefore we cluster subproblems according to bounding box size and solve these batches sequentially.
In Figure \ref{fig:clustering} we show an example of this clustering, as well as the dependence of runtime with the number of clusters for several resolutions (computed on a subset of our problem data). We obtained the best results using a single cluster for $N \le 256$ and then scaling the number of clusters linearly in $N$. We will use this heuristic criterion for our subsequent experiments.

\begin{figure}[H]
	\centering
	\includegraphics[width=0.9\linewidth]{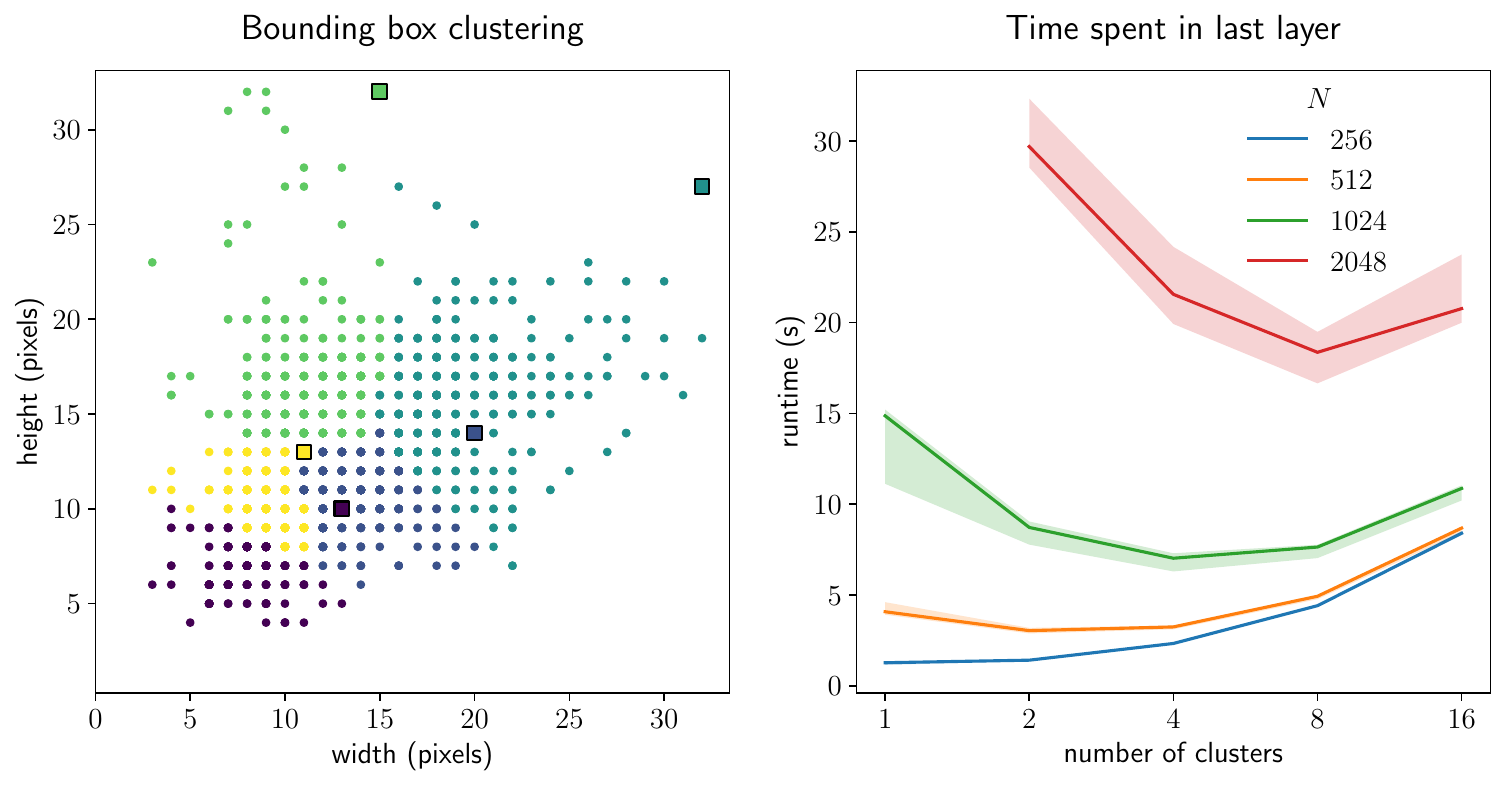}
	\caption{Left, every point represents the width and height (in numbers of pixels) of one composite cell $Y$-marginal, in a problem with $N = 1024$, $s = 4$. The coloring represents the label after performing a K-means clustering. The squared markers indicate the dimensions of the resulting bounding boxes. Right, time spent in the last \DomDecGPU\ multiscale layer, depending of the number of clusters. The solid line represents the median time, while the shading extent represents the range between the 0.25 and 0.75 quantiles. The optimal number of clusters approximately seems to be given by $N / 256$.}
	\label{fig:clustering}
\end{figure}

\subsubsection{Results}

\paragraph{Runtime.}
Results of the numerical experiments are summarized in Table \ref{tab:benchmark} and the total runtimes are visualized in Figure \ref{fig:total-time-benchmark}.
As can be seen, on large \SinkhornGPU\ struggles to solve large problems.
Increasing the error tolerance for the marginal in the stopping criterion for \SinkhornGPU\ error shifts the runtime curve but does not change the scaling with respect to $N$.
In contrast, both domain decomposition implementations exhibit a better scaling of the runtime with respect to $N$.
The runtime of \DomDecMPI\ (with 16 workers) is approximately linear in the number of pixels $N^2$, as reported \cite{BoSch2020}.
The new implementation \DomDecGPU\ is consistently faster and by far the fastest algorithm on the largest test problems by approximately one order of magnitude. The optimal transport between two megapixel images is computed in approximately 11 seconds. In the observed range the runtime scales sublinearly in $N^2$. It seems possible that this is due to the overhead of parallelization (similar to \DomDecMPI\ on small problems) and that it will transition to approximately linear scaling on even larger problems.

As reported for \DomDecMPI\ in \cite{BoSch2020}, \DomDecGPU\ spends most time on solving the cell problems (around 70-80\%). The remaining time is mainly spent mostly on handling the cell marginals in the methods $\textsc{BasicToComposite}$, $\textsc{CompositeToBasic}$, and $\textsc{BalanceGPU}$. The refinement between multiscale layers and the clustering of the cell problems are relatively cheap steps, using up respectively only 1\% and 2\% of the total runtime for the finest resolutions.

\begin{table}[h]
	\begin{tabular}{lrrrrrr}
\toprule
image size  &    64$\times$64   &    128$\times$128  &    256$\times$256  &    512$\times$512  &     1024$\times$1024 &    2048$\times$2048 \\
\midrule
\multicolumn{7}{l}{total runtime (s)}
\\
\midrule
\DomDecGPU &    1.68 &    2.05 &    3.08 &     5.9 &     11.3 &    29.3 \\
\DomDecMPI &    1.51 &    3.05 &    8.52 &    30.5 &      130 &     561 \\
\SinkhornGPU &    1.73 &    2.77 &    7.08 &    53.7 &      483 &     - \\
\SinkhornGPU\ (Err = $10^{-3}$) &   0.407 &   0.569 &    1.29 &    8.43 &     69.4 &     - \\
\midrule
\multicolumn{7}{l}{time spent in \DomDecGPU's subroutines (s)}
\\
\midrule
sinkhorn subsolver &    1.21 &    1.38 &    2.21 &    4.56 &     9.06 &    22.5 \\
bounding box processing &   0.173 &   0.275 &   0.341 &   0.523 &    0.992 &    3.98 \\
balancing &   0.044 &  0.0752 &    0.15 &   0.351 &    0.644 &    1.38 \\
minibatching and clustering  &   0.105 &   0.127 &   0.151 &   0.183 &     0.24 &   0.417 \\
multiscale refining & 2.2e-02 & 2.9e-02 &  0.0379 &  0.0561 &   0.0985 &   0.296 \\
 \midrule
 \multicolumn{7}{l}{\DomDecGPU\ solution quality}
 \\
 \midrule
$L^1$ marginal error $Y$ & 7.5e-12 & 3.5e-11 & 1.5e-10 & 6.1e-10 &  2.5e-09 & 1.0e-08 \\
$L^1$ marginal error $X$ & 9.2e-05 & 9.3e-05 & 9.3e-05 & 9.6e-05 &  9.3e-05 & 1.1e-04  \\
relative primal-dual gap & 4.2e-05 & 2.6e-05 & 3.9e-05 & 2.0e-05 &  1.4e-05 & 1.2e-05 \\
\midrule
\multicolumn{7}{l}{\SinkhornGPU\ solution quality}
\\
\midrule
$L^1$ marginal error $X$ & 1.0e-04 & 1.0e-04 & 1.0e-04 & 1.0e-04 &  1.0e-04 &     - \\
relative dual score & 3.3e-05 & 2.1e-05 & 3.0e-05 & 1.5e-05 &   1.0e-05 &     - \\
\bottomrule
\end{tabular}

	\caption{Summary of the average performance for the examined solvers. For each image size results are averaged over 45 problem instances (Section \ref{sec:gpu-domdec}).}
	\label{tab:benchmark}
\end{table}

\paragraph{Sparsity.}
In \DomDecMPI\ the basic cell marginals are stored in a sparse data structure. As outlined above, as GPUs work most efficiently on contiguous data, \DomDecGPU\ uses a bounding box data structure. While this can still benefit from sparsity, it is not as memory efficient as in \DomDecMPI.
As shown in Figure \ref{fig:boundingbox}, the bounding boxes inflict some memory overhead that is increasing with the problem size.
While the number of non-zero entries increases approximately as the marginal size $N^2$ (as observed for \DomDecMPI\ in \cite{BoSch2020}), the number of stored entries grows faster approximately by a factor $\sqrt{N}$.
This is currently the limiting factor for further increasing the problem size for \DomDecGPU. This can presumably be improved by more sophisticated clustering and batching of the cell problems.

\begin{figure}[h]
	\centering
	\includegraphics[width=0.9\linewidth]{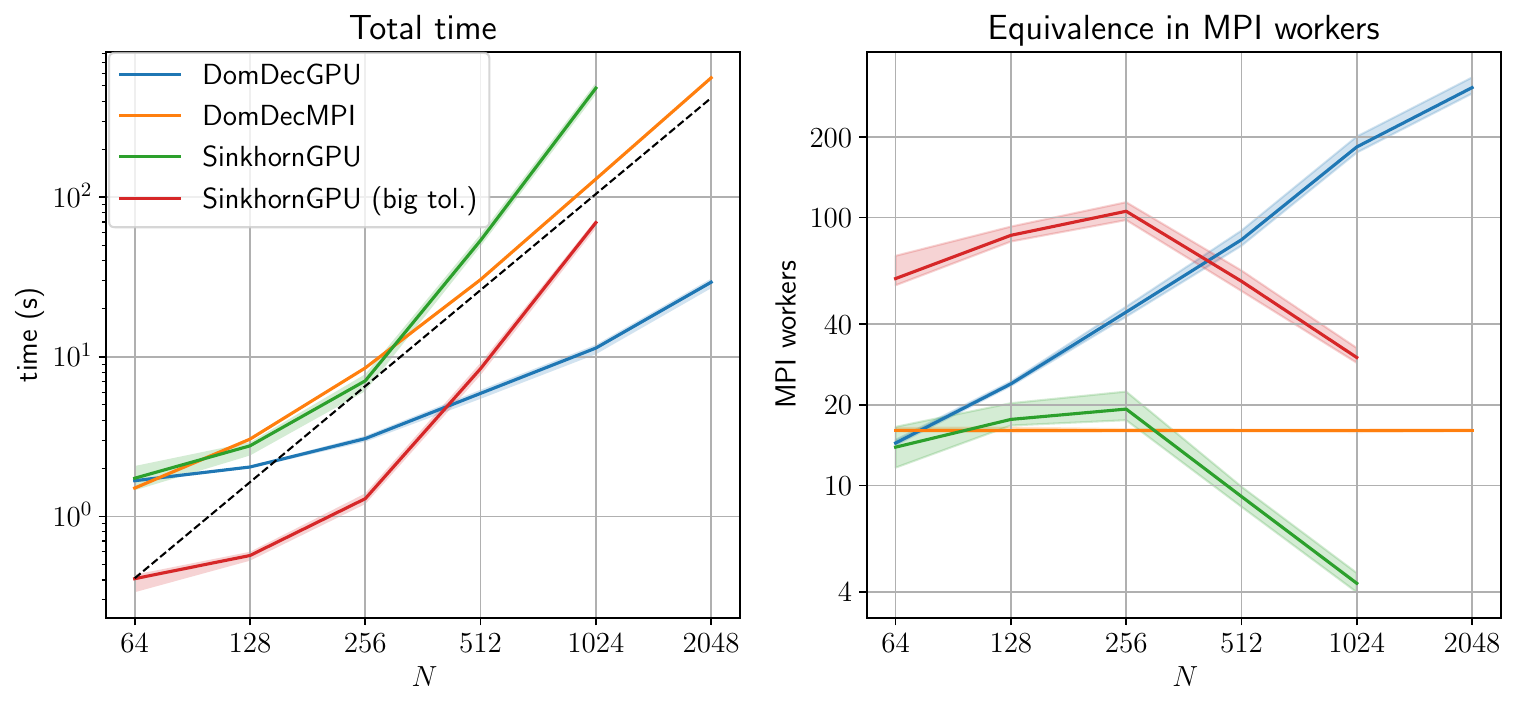}
	\caption{Left, runtime comparison between different solvers. The solid line and shading represent the median and the 0.25 and 0.75 quantiles. The \textsc{SinkhornGPU} solver with higher error tolerance has a looser $X$-marginal error objective of $\rm Err = 10^{-3}$.
		The dashed line is proportional to the marginal size, i.e.~$N^2$.
		Right, runtimes normalized to the \DomDecMPI\ runtime to show a rough equivalence of the performance in MPI workers.}
	\label{fig:total-time-benchmark}
\end{figure}

\begin{figure}[H]
	\centering
	\includegraphics[width=0.9\linewidth]{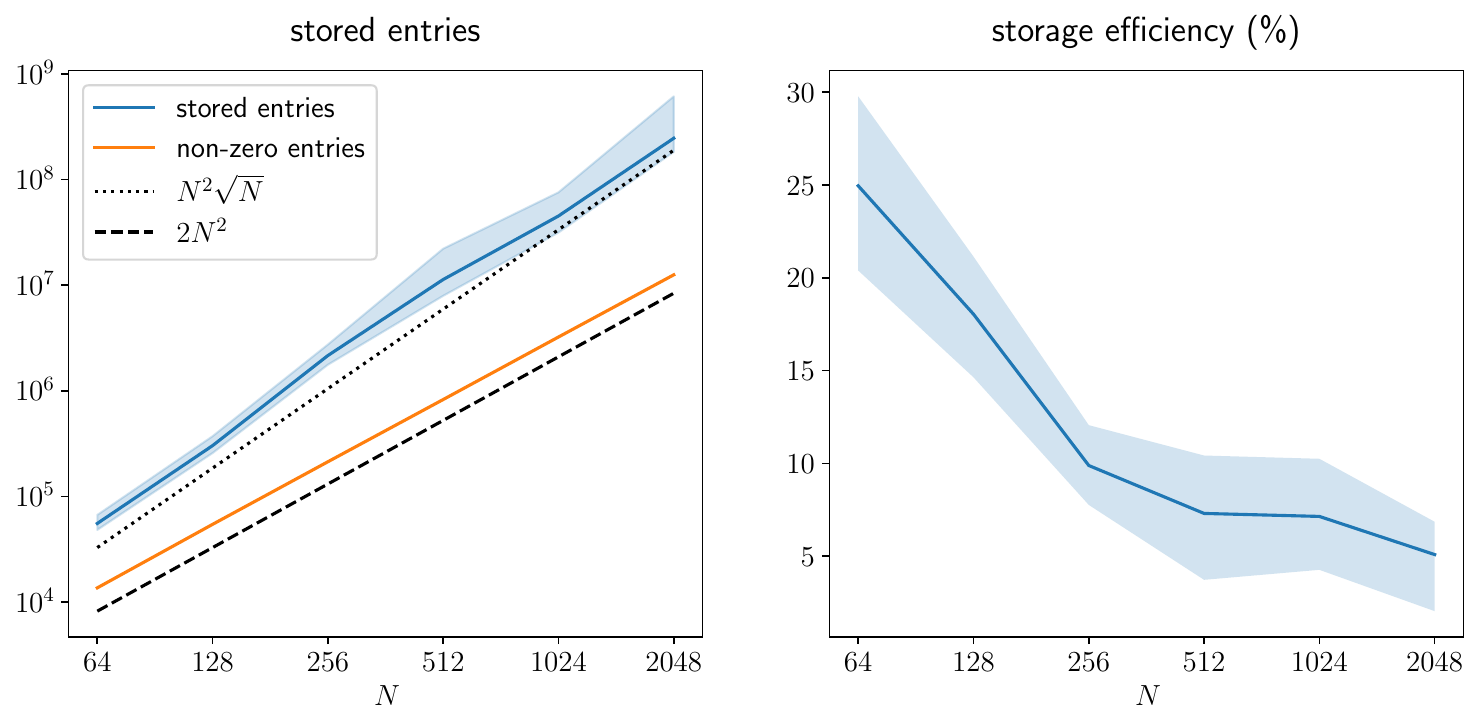}
	\caption{Left, number of stored and non-zero entries in the final plan. The number of non-zero entries grows linearly with the problem size $N^2$, while the dimensions of the bounding box structure grow slightly faster. Right, storage efficiency (proportion of stored entries that are positive) over resolution.}
\end{figure}

\paragraph{Solution quality.} We evaluate the solution quality of \DomDecGPU\ by measuring the violation of the transport plan marginal constraints and with the relative primal-dual gap
\begin{equation}
\frac{E(\pi) - D(\alpha^\dagger, \beta^\dagger)}{D(\alpha^\dagger, \beta^\dagger)},
\quad \tn{with }
\begin{cases}
E(\pi) \assign 
\la c, \pi \ra 
+
\veps \KL(\pi | \mu \otimes \nu), 
\\[2mm]
D(\alpha, \beta) \assign 
\la \alpha, \mu \ra 
+ 
\la \beta, \nu \ra 
+
\veps \left\langle 1 - \exp\left(\tfrac{\alpha \oplus \beta - c}{\veps}\right), \mu \otimes \nu \right\rangle.
\end{cases}
\end{equation}
where $\alpha^\dagger$ and $\beta^\dagger$ are the dual domain decomposition potentials, which are obtained by applying the stitching procedure described in \cite[Section 6.3]{BoSch2020}.

For comparison with the \SinkhornGPU\ solver we use the relative dual score
\begin{equation}
\frac{D(\alpha^*, \beta^*) - D(\alpha^\dagger, \beta^\dagger) }{D(\alpha^\dagger, \beta^\dagger) },
\end{equation}
where $\alpha^*$ and $\beta^*$ the final dual potentials obtained by \SinkhornGPU.

We observe that \DomDecGPU\ generates solutions of high quality. The relative PD gap is always smaller than $10^{-5}$. The idealized algorithm would by design perfectly satisfy the $Y$-marginal constraints. Deviations are purely due to finite floating point precision. Indeed, numerically the violation is very small, growing with the problem size, and is still below $10^{-8}$ on the largest test problems. The $X$ marginal error is approximately $10^{-4}$, consistent with the stopping criterion for the Sinkhorn algorithm on the cell problems.

Upon convergence with $\tn{Err}=10^{-4}$, the relative primal dual gap is slightly positive, indicating that the global Sinkhorn solver produced somewhat more precise dual variables. But this deviation is on the order of the PD gap of \DomDecGPU\ and thus not substantial.
However, as shown above this small increase in precision comes a a substantial cost in terms of runtime.

\paragraph{Conclusion.} Previous work has shown that global Sinkhorn solvers struggle to achieve a good precision in large optimal transport problems, and that domain decomposition is an efficient alternative. Our new GPU implementation of domain decomposition outperforms the previous CPU implementation and a monolithic GPU Sinkhorn implementation. On megapixel-sized images, \DomDecGPU\ achieves a performance comparable to more than 200 MPI workers using a single GPU.
An important task for future work is an improved memory efficiency for the bounding box data structure and subproblem clustering.

\end{document}